\documentclass{article}

\usepackage{psfrag}
\usepackage{graphicx}
\usepackage{amsmath,amsfonts,amssymb,amsxtra,amsthm,mathrsfs,enumerate}
\usepackage{pinlabel}
\usepackage{color}
\usepackage{tikz-cd}
\usepackage{hyperref}
\usepackage{cancel}
\usepackage[normalem]{ulem}
\usepackage{wrapfig}

\tikzset{
math to/.tip={Glyph[glyph math command=rightarrow]},
loop/.tip={Glyph[glyph math command=looparrowleft, swap]},
loop'/.tip={Glyph[glyph math command=looparrowleft]},
 weird/.tip={Glyph[glyph math command=Rrightarrow, glyph length=1.5ex]},
  pi/.tip={Glyph[glyph math command=pi, glyph length=1.5ex, glyph axis=0pt]},
}

\DeclareMathOperator{\cl}{cl}
\DeclareMathOperator{\scl}{scl}

\DeclareMathOperator{\Wh}{Wh}

\def\immerses{\looparrowright}
 \def\into{\hookrightarrow}

\def\ncl#1{\mathord{\langle}\mskip -4mu plus 0mu minus 0mu
  \mathord{\langle}#1\mathord{\rangle}\mskip -4mu plus 0mu minus 0mu
  \mathord{\rangle}}

\def\real#1{\boldsymbol #1}

\def\QQ{\mathbb{Q}}

\def\RR{\mathbb{R}}
\def\NN{\mathbb{N}}
\def\ZZ{\mathbb{Z}}

\def\calE{\mathcal{E}}
\def\calP{\mathcal{P}}

\newtheorem{theorem}{Theorem}[section]

\newtheorem{lemma}[theorem]{Lemma}
\newtheorem{corollary}[theorem]{Corollary}

\newtheorem{proposition}[theorem]{Proposition}
\newtheorem{conjecture}[theorem]{Conjecture}
\newtheorem{question}[theorem]{Question}
\newtheorem{introthm}{Theorem}
\newtheorem{introcor}[introthm]{Corollary}

\theoremstyle{definition}
\newtheorem{definition}[theorem]{Definition}

\theoremstyle{remark}
\newtheorem{remark}[theorem]{Remark}

\newtheorem{example}[theorem]{Example}

\author{Larsen Louder and Henry Wilton}

\newcommand{\Addresses}{{
  \bigskip
  \footnotesize

  L.~Louder, \textsc{Department of Mathematics, University College London, Gower Street, London WC1E  6BT, UK}\par\nopagebreak
  \textit{E-mail address:} \texttt{l.louder@ucl.ac.uk}

  \medskip

  H. Wilton, \textsc{DPMMS, Centre for Mathematical Sciences, Wilberforce Road, Cambridge CB3 0WB, UK}\par\nopagebreak
  \textit{E-mail address:} \texttt{h.wilton@maths.cam.ac.uk}

}}

\title{Uniform negative immersions and the coherence of one-relator groups}

\begin{document}
\maketitle

\begin{abstract}
Previously, the authors proved that the presentation complex of a one-relator group $G$ satisfies a geometric condition called \emph{negative immersions} if every two-generator, one-relator subgroup of $G$ is free.  Here, we prove that one-relator groups with negative immersions are coherent, answering a question of Baumslag in this case.  Other strong constraints on the finitely generated subgroups also follow such as, for example, the co-Hopf property. The main new theorem strengthens negative immersions to \emph{uniform negative immersions}, using a rationality theorem proved with linear-programming techniques.
\end{abstract}


\section{Introduction}

\begin{definition}[Coherence]
  A group $G$ is said to be \emph{coherent} if all its finitely generated subgroups are finitely presentable.
\end{definition}

A notorious question of Baumslag  \cite[p.\ 76]{baumslag_problems_1974} asks whether every one-relator group $G=F/\ncl{w}$ is coherent.  It is a curious fact that the many known examples of one-relator groups with pathological propertes --- for instance, Baumslag--Solitar groups, the Baumslag--Gersten group \cite{baumslag_non-cyclic_1969}, or the recent examples of Gardam--Woodhouse \cite{gardam_geometry_2019} --- all have two generators. Our first theorem resolves Baumslag's question, as long as such subgroups are excluded.

\begin{introthm}
  \label{thm: Coherence}
  Let $G=F/\ncl{w}$ be a one-relator group. If every two-generator, one-relator subgroup of $G$ is free, then $G$ is coherent.
\end{introthm}

The case where $G$ has torsion was proved independently by the authors \cite{louder_one-relator_2020} and Wise \cite{wise_coherence_2020}. Sapir and Spakulova showed that a generic one-relator group with at least three generators is coherent \cite{sapir_almost_2011}.

The hypothesis of Theorem \ref{thm: Coherence} --- that every two-generator, one-relator subgroup is free --- appears on its face to be difficult to check. However, the results of \cite{louder-wilton2} imply that it is equivalent to other, more effective, conditions.  The following theorem, which follows from \cite[Theorems 1.3 and 1.5]{louder-wilton2}, summarises these equivalences. (The relevant definitions are explained below.)  By convention, the relator $w$ is always assumed to be non-trivial.

\begin{theorem}[\cite{{louder-wilton2}}]\label{thm: Negative immersions}
Let $G=F/\ncl{w}$ be a one-relator group, and let $X$ be the natural presentation complex. The following conditions are equivalent:
\begin{enumerate}[(i)]
\item the primitivity rank $\pi(w)>2$;
\item $X$ has negative immersions;
\item every two-generator subgroup of $G$ is free;
\item every two-generator, one-relator subgroup of $G$ is free.
\end{enumerate}
\end{theorem}
Since it is the briefest to state, if any of these hold we say that $X$ or $G$ has \emph{negative immersions}.

\begin{remark}\label{rem: Puder, Sapir--Spakulova, Arzhantsva--Olshanskii}
Puder's primitivity rank $\pi(w)$ is a positive integer (or infinite if $w$ is a basis element) \cite{puder_primitive_2014}. We will not need its definition here, but we make the following remarks.
\begin{enumerate}[(i)]
\item The primitivity rank $\pi(w)$ can be computed algorithmically: see \cite[Lemma 6.4]{louder-wilton2} or \cite[Appendix A]{puder_primitive_2014} for algorithms, and see \cite{cashen_short_2020} for extensive computations.
\item Puder also showed that a generic word $w$ in a free group of rank $r$ has $\pi(w)=r$ \cite[Corollary 8.3]{puder_expansion_2015}, so generic one-relator groups with more than two generators have negative immersions. (A more refined result was later proved by Kapovich \cite{kapovich_primitivity_2022}.)
\item Combining \cite[Corollary 8.3]{puder_expansion_2015} with Theorems \ref{thm: Coherence} and \ref{thm: Negative immersions} recovers the result of Sapir--Spakulova that random one-relator groups with at least three generators are coherent. Instead of \cite[Corollary 8.3]{puder_expansion_2015}, one can also appeal to the work of Arzhantseva--Olshanskii, who proved that a random group with at least 3 generators is 2-free \cite{arzh-olsh}.
\end{enumerate}
\end{remark}

Condition (ii) of Theorem \ref{thm: Negative immersions} is a geometric condition on the 2-complex $X$.  An \emph{immersion}  is a combinatorial map of 2-complexes that is locally injective.  Roughly, a 2-complex $X$ has \emph{negative immersions} if, for any finite, connected 2-complex $Y$ immersing into $X$, either
\begin{enumerate}[(i)]
\item $\chi(Y)<0$, or
\item $Y$ is homotopic to a graph.
\end{enumerate}
See Definition \ref{def: Negative immersions} for a precise definition. As usual, $\chi(Y)$ denotes the Euler characteristic of $Y$.  The negative immersions property is motivated by an analogy with the concept of \emph{non-positive immersions}, which is defined similarly but with (i) replaced by the inequality $\chi(Y)\leq 0$. The presentation complex of any torsion-free one-relator group has non-positive immersions \cite{helfer-wise,louder-wilton}.

The proof of Theorem \ref{thm: Coherence} leads to various other strong constraints on the subgroups of one-relator groups with negative immersions. The next theorem summarises these.

\begin{introthm}
  \label{thm: Other subgroup constraints}
  Let $G=F/\ncl{w}$ be a one-relator group with negative immersions.
  \begin{enumerate}[(i)]
   \item Every finitely generated, {one-ended} subgroup $H$ of  $G$ is co-Hopfian, i.e.\ $H$ is not isomorphic to a proper subgroup of itself. In particular, {if $G$ itself is one-ended then it} is co-Hopfian.
   \item For any integer $r$, there are only finitely many conjugacy classes of {finitely generated, one-ended} subgroups $H$ of $G$  such that the abelianisation of $H$ has rational rank at most  $r$.
    \item Every finitely generated {non-cyclic} subgroup $H$ of $G$ is large in the sense of Pride, i.e.\ there is a subgroup $H_0$ of finite index in $H$ that surjects a non-abelian free group.
\end{enumerate}
\end{introthm}

{In an earlier paper, the authors conjectured that every one-relator group $G$ with negative immersions is hyperbolic \cite[Conjecture 1.9]{louder-wilton2}. Theorem \ref{thm: Other subgroup constraints} proves various consequences of this conjecture. Specifically, item (i) is consistent with Sela's theorem that one-ended hyperbolic groups are co-Hopfian \cite[Theorem 4.4]{sela_structure_1997}, while item (ii) is a stronger version of a subgroup rigidity theorem of Gromov \cite[5.3.C']{gromov_hyperbolic_1987}, Rips--Sela \cite[Theorem 7.1]{rips_structure_1994} and Delzant \cite{delzant_limage_1995}. Note, however, that items (ii) and (iii) are not implied simply by the fact that $G$ is hyperbolic.  

Recently, Linton has proved the authors' conjecture \cite[Theorem 8.2]{linton_one-relator_2022}, making use of Theorem \ref{thm: Other subgroup constraints} in his proof. Linton's theorem is especially remarkable in light of the recent discovery by Italiano--Martelli--Migliorini of higher-dimensional non-hyperbolic groups of finite type that do not contain Baumslag--Solitar subgroups \cite[Corollary 3]{italiano_hyperbolic_2023}.}


The main new technical ingredient in the proof of Theorem \ref{thm: Coherence} is a uniform version of negative immersions, which provides a negative upper bound for a normalised version of the Euler characteristic. Roughly, a 2-complex $X$ has \emph{uniform negative immersions} if there is $\epsilon>0$ such that for any finite, connected 2-complex $Y$ immersing into $X$, either
\begin{enumerate}[(i)]
\item 
\[
\frac{\chi(Y)}{\#\{2\mathrm{-cells\,of}\,Y\}}\leq -\epsilon\,,
\]
or
\item $Y$ can be simplified by a homotopy. (In the language developed below, $Y$ is \emph{reducible}.)
\end{enumerate}
The reader is referred to Definition \ref{def: Uniform negative immersions} for a precise definition.

\begin{remark}\label{rem: Wise's negative immersions}
In Wise's work on the coherence of one-relator groups with torsion \cite{wise_coherence_2020} (and also implicitly in the authors' \cite{louder_one-relator_2020}), a key role is played by a property that Wise also calls `negative immersions', but that is closer in spirit to (indeed, stronger than) the notion of \emph{uniform} negative immersions used here. See Section \ref{sec: Wise's notion of sectional curvature} for a comparison between the notion of (uniform) negative immersions used in this paper and Wise's notion.
\end{remark}

The main new technical ingredient towards the proof of Theorem \ref{thm: Coherence} is the following result, which establishes uniform negative immersions for the 2-complexes of interest to us.

\begin{introthm}\label{thm: Uniform negative immersions}
Let $X$ be the presentation complex of a one-relator group. If $X$ has negative immersions then it has uniform negative immersions.
\end{introthm}

{The deduction of Theorem \ref{thm: Coherence} from Theorem \ref{thm: Uniform negative immersions} is similar to \cite{louder_one-relator_2020} and \cite{wise_coherence_2020}, with an additional lemma (Lemma \ref{lem: Irreducible core}) to account for the subtle nature of item (ii) in the definition.

In the terminology of \cite{wilton_rational_2022}, Theorem \ref{thm: Uniform negative immersions} asserts that one-relator presentation complexes  with negative immersions have \emph{negative irreducible curvature}.}

Theorem \ref{thm: Uniform negative immersions} follows from a rationality theorem, Theorem \ref{thm: Rationality theorem}, similar to those that appear in Calegari's proof of the rationality of stable commutator length in free groups \cite{calegari_stable_2009} and in the second author's work on surface subgroups of graphs of free groups \cite{wilton_essential_2018}. 

The first idea is to expand the category of maps considered from immersions $Y\immerses X$ to the more general class of \emph{face-essential} maps $f:Y\to X$. This class of maps allows branching in the centre of the 2-cells, and thus each 2-cell $C$ of $Y$ comes with a well-defined \emph{degree} $\deg_C(f)$. The degree of the map $f$ is then the sum of the degrees of the 2-cells of $Y$, and we adjust the relevant quantities to take it into account. The role of Euler characteristic is taken by the \emph{total curvature}
\[
\tau(Y):=\deg(f)+\chi(Y^{(1)})
\]
where $\chi(Y^{(1)})$ is the Euler characteristic of the 1-skeleton of $Y$, and the role of the number of 2-cells is taken by the degree $\deg(f)$. Note that, when $f$ is an immersion, these reduce to the usual quantities: $\tau(Y)=\chi(Y)$, while $\deg(f)$ is the number of 2-cells of $Y$.  With these adjustments, the supremum of the quantity
\[
\frac{\tau(Y)}{\deg(f)}
\]
over all face-essential maps $f:Y\to X$ with $Y$ irreducible can be computed as a solution to a rational linear programming problem; this is the content of Theorem \ref{thm: Rationality theorem}. 

In particular, this supremum is in fact a maximum. The results of \cite{louder-wilton2} then imply that this maximum is strictly negative, so provides the uniform bound $-\epsilon$ needed to prove Theorem \ref{thm: Uniform negative immersions}.

This proof provides no information about the value of $\epsilon$, but we make the following conjecture. A face-essential map of 2-complexes $f:Y\to X$ is called \emph{essential} if, for each component of $Y$, the induced map on 1-skeleta induces an injective homomorphism on fundamental groups.

\begin{conjecture}\label{conj: Stability conjecture}
Let $G=F/\ncl{w}$ be a one-relator group and let $X$ be the associated presentation complex. The supremum of the quantity
\[
\frac{\tau(Y)}{\deg(f)}
\]
across all {essential} maps $f$ from irreducible 2-complexes $Y$ to $X$ is equal to $2-\pi(w)$.
\end{conjecture}

In an earlier version of this paper, we stated a stronger version of Conjecture \ref{conj: Stability conjecture} for the supremum across all face-essential maps. The next example shows that this stronger conjecture is false.

\begin{example}\label{eg: Unstable face-essential map}
Let $X$ be the presentation complex of 
\[
\langle a,b,c,d\mid a^4b^2c^2d^2a^3b^2c^2d^2\rangle
\]
and let $Y$ be the presentation complex of
\[
\langle x,y,z\mid x^2yzyxz\rangle\,.
\]
Write $w$ for the relator of $X$ and $u$ for the relator of $Y$. The homomorphism of free groups defined by
\[
x,y\mapsto a\,,\,z\mapsto ab^2c^2d^2
\]
sends $u$ to $w$ and so, after subdividing the 1-skeleton of $Y$ appropriately, defines a combinatorial map of 2-complexes $f:Y\to X$. Since this morphism sends the single 2-cell of $Y$ homeomorphically to the single 2-cell of $X$, $f$ is face essential (see Definition \ref{def: Face-essential map}) and $\deg(f)=1$. Furthermore, the Whitehead graph of $Y$ consists of two squares glued along an edge, so $Y$ is (visibly) irreducible (see Definition \ref{def: Visibly reducible and visibly irreducible}). 

However,
\[
\frac{\tau(Y)}{\deg(f)}=\frac{\deg(f)+\chi(Y_{(1)})}{\deg(f)}=-1
\]
while a calculation using Puder's algorithm shows that $\pi(w)=4$ \cite{cashen_personal_2023}. Thus, $f$ provides an example of a face-essential map from an irreducible 2-complex $Y$ such that
\[
\frac{\tau(Y)}{\deg(f)}>2-\pi(w)\,.
\]
In particular, the `essential' hypothesis in Conjecture \ref{conj: Stability conjecture} cannot be relaxed to `face-essential'.
\end{example}



Conjecture \ref{conj: Stability conjecture} has precursors in the literature. The following conjecture of Heuer posits a surprising relationship between stable commutator length and Puder's primitivity rank \cite[Conjecture 6.3.2]{heuer_constructions_2019}. (The same conjecture was made independently by Hanany and Puder \cite[Conjecture 1.14]{hanany_word_2020}.)

\begin{conjecture}[Heuer's conjecture]\label{conj: Heuer's conjecture}
If $w$ is a non-trivial element of the commutator subgroup of a free group $F$ then
\[
2\scl(w)\geq\pi(w)-1\,.
\] 
\end{conjecture}

Conjecture \ref{conj: Stability conjecture} implies Heuer's conjecture, and indeed Theorem \ref{thm: Negative immersions}, together with Calegari's rationality theorem for stable commutator length \cite{calegari_stable_2009}, unconditionally implies a consequence of Heuer's conjecture.

\begin{introcor}\label{cor: Weak Heuer}
If $w$ is an element of the commutator subgroup of a free group $F$ and $\pi(w)>2$ then
\[
2\scl(w)>1\,.
\]
\end{introcor}

For context, Duncan--Howie proved that $\scl(w)\geq 1/2$ for any non-trivial element of $[F,F]$ \cite{duncan-howie}. The statement of Corollary \ref{cor: Weak Heuer} was noted by Heuer and L\"oh \cite[Question 1.3(4)]{heuer_simplicial_2022}, but we will give a complete proof  in \S\ref{sec: Stable commutator length} below. The proof of Corollary \ref{cor: Weak Heuer} and the deduction of Corollary \ref{conj: Heuer's conjecture} from Conjecture \ref{conj: Stability conjecture} both use Proposition \ref{prop: Essential map realising scl}, which summarises an argument of Calegari \cite[Lemma 2.7]{calegari_surface_2008}.

The paper is structured as follows. Section \ref{sec: Preliminaries} is devoted to setting up the framework in which we describe 2-complexes and maps between them. It is convenient to work with \emph{pre-complexes}, which describe pieces that can be glued together to form complexes.  Section \ref{sec: Irreducible complexes} describes the notion of \emph{irreducible complexes} that plays a crucial role in our definitions, and enables us to define (uniform) negative immersions.  In Section  \ref{sec: A linear system}, we set up the linear system needed to prove the rationality theorem, from which Theorem \ref{thm: Uniform negative immersions} follows.  The brief Section \ref{sec: The case with torsion} explains why one-relator groups with torsion also have uniform negative immersions. Finally, theorems about the the subgroup structure of the fundamental groups of 2-complexes with uniform negative immersions are proved in Section \ref{sec: Uniform negative immersions and subgroups}, and Theorems \ref{thm: Coherence} and \ref{thm: Other subgroup constraints} follow immediately.

\subsection*{Acknowledgements}

The authors would like to thank Chris Cashen for computing the
primitivity rank of Example \ref{eg: Unstable face-essential map} and
the anonymous referee for numerous suggested improvements to the
exposition. Louder also thanks Mladen Bestvina for explaining the
tower, Wise's ideas, and the Scott lemma to him back in Winter
2004.

\section{Preliminaries}\label{sec: Preliminaries}

In this section, we give some foundational definitions in order to fix notation.

\subsection{Graphs and pre-graphs}
 
 Graphs are our fundamental objects of study. A subgraph is always a closed subspace in any of the standard combinatorial frameworks, but it is also convenient to discuss open sub-objects of graphs. We therefore work with the slightly more general notion of \emph{pre-graphs}.

\begin{definition}[Pre-graph]\label{def: Pre-graph}
  A \emph{pre-graph} is a tuple
  \[
  G=(V_G,E_G,\iota_G,\tau_G)=(V,E,\iota,\tau)
  \]
  where:
  \begin{enumerate}[(i)]
  \item $V$ and $E$  are sets (the \emph{vertices} and \emph{edges} of $G$, respectively);
  \item $E^{\iota}_G=E^\iota$  and $E^{\tau}_G=E^\tau$ are subsets of $E$; and
  \item  $\iota\colon E^{\iota}\to V$  and $\tau\colon E^{\tau}\to V$ are maps.
  \end{enumerate}
 The subscripts will be suppressed only when there is no danger of confusion. A \emph{morphism} $f$ from a pre-graph $G$ to a pre-graph $H$ consists of maps $f:V_G\to V_H$ and $f:E_G\to E_H$ such that $f(E_G^\iota)\subseteq E_H^\iota$ and $f(E_G^\tau)\subseteq E_H^\tau$, and such that $\iota_H\circ f=f\circ \iota_G$ and $\tau_H\circ f=f\circ \tau_{{G}}$.
\end{definition}

The \emph{boundary} of a pre-graph $G$ is defined to be
\[
\partial G = E\setminus E^{\iota}\cap E^{\tau}\,.
\]

\begin{definition}[Graph]\label{def: Graph}
  {An (oriented)} \emph{graph} is a pre-graph with empty boundary.
\end{definition}

The reader is referred to the classic paper of Stallings \cite{stallings-folding} for the standard theory of the topology of graphs. Stallings phrases his arguments in terms of Serre graphs, which differ slightly from the model of graphs we adopt here, but his arguments are easily adapted to our context and we leave the details to the reader.  Stallings does not use pre-graphs, but again the theory is easy to adapt.

\begin{definition}[Immersion]\label{def: Immersion}
 A morphism $f\colon G\to H$ of pre-graphs is an \emph{immersion} if, for
  all edges $e$ and $e'$ of $G$, if
  $\alpha(e)=\alpha(e')$ for some $\alpha\in\{\iota_G,\tau_G\}$ and
  $f(e)=f(e')$, then $e=e'$. An immersion $f\colon G\to H$ of pre-graphs is \emph{closed} if $f(\partial G)\subseteq\partial H$.
\end{definition}
  
 The standard fiber product construction for graphs (see \cite[1.3]{stallings-folding}) also makes sense for pre-graphs.

\begin{definition}\label{def: Fiber product}
Consider a pair of morphisms of pre-graphs
\[
f_i:G_i\to H
\]
where $i=1,2$.  The \emph{fiber product}  $G_1\times_H G_2$ has vertex (respectively edge) set defined to be the subset of $V_{G_1}\times V_{G_2}$ (respectively, $E_{G_1}\times E_{G_2}$) on which $f_1$ and $f_2$ agree. The attaching maps are then defined in the natural way wherever they make sense.
\end{definition}

This definition coincides with the standard definition whenever $G_1$, $G_2$ and $H$ are graphs.  It enjoys various well-known properties, which are left as easy exercises: it is a pullback in the category of pre-graphs, the pullback of an immersion is an immersion, etc.

\subsection{Complexes and pre-complexes}

Our graphs will often be the 1-skeleta of 2-dimensional complexes, which can be formalised as immersions from disjoint unions of circles to graphs, where the circles correspond to the boundaries of the 2-cells (which we will always call \emph{faces}). As in the case of graphs, we will sometimes want to discuss open sub-objects, and we therefore introduce \emph{pre-complexes}. In this case, slightly more general pre-graphs can arise as the boundaries of faces.

\begin{definition}\label{def: Linear pre-graph}
 A connected finite pre-graph is a \emph{pre-cycle} if every vertex has exactly two incident edges. A pre-cycle is a \emph{cycle} if it has empty boundary, and is an \emph{open arc} otherwise.
\end{definition}

We can now define a pre-complex using immersions of pre-cycles.

\begin{definition}\label{def: pre-complex}
  A \emph{(2-dimensional) pre-complex} $X$ is a tuple
  \[
  (G_X,S_X,w_X)
  \]
  where
  \[
  G_X = (V_X,E_X,\iota_X,\tau_X)
  \]
is a pre-graph, $S_X$ is a union of pre-cycles, and $w_X\colon S_X\immerses G_X$ is a closed immersion. The pre-graph $G_X$ is the  \emph{1-skeleton} of $X$, $S_X$ is the collection of \emph{{(pre-)faces}} of $X$ and $w_X$ is the \emph{attaching map}. For brevity, we will suppress the subscript when there is no danger of confusion.

A pre-complex $X$ is a \emph{(2-dimensional) complex} if the 1-skeleton $G_X$ is a graph. {Note that, in this case,} every component of $S_X$ is a cycle.
\end{definition}

\begin{figure}[ht]
  \centerline{
    \includegraphics[width=.8\textwidth]{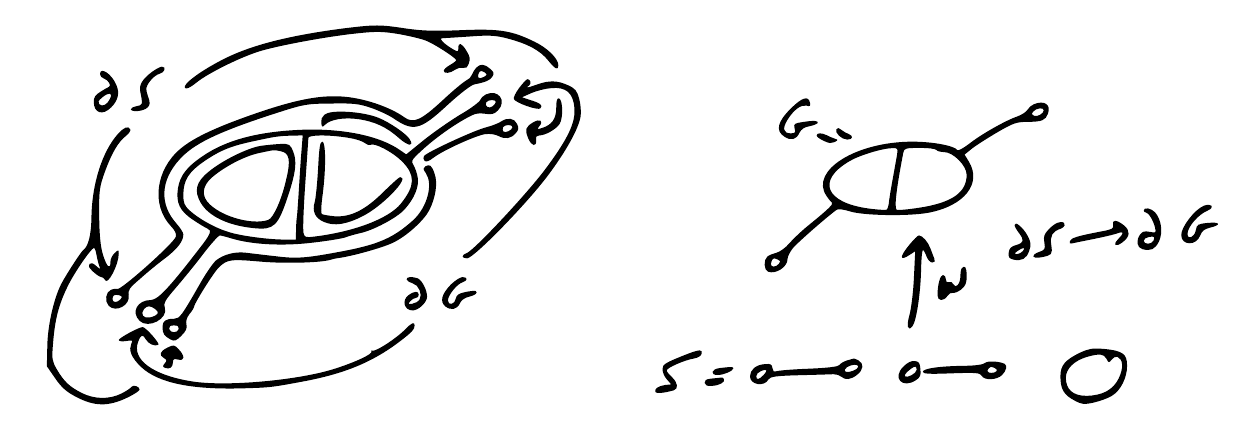}
  }
  \caption{A pre-complex $(G,S,w)$: $G$ is a theta graph with two
    half-edges sticking off, and $S$ is the union of two open arcs and one
    cycle. The map $w$ is a closed immersion, since it maps $\partial S$ to
    $\partial G$.}
  \label{fig:two_complex}
\end{figure}

Graphs and complexes have natural topological realisations, which we denote here with bold letters: the realisation of a graph $G$ is a topological space denoted by $\real G$ and the realisation of a complex $X$ is a topological space denoted by $\real X$. In both cases, a morphism $f$ functorially induces a continuous map of realisations $\real f$.

\subsection{Asterisks and regular neighbourhoods}

It is often fruitful to study graphs and 2-complexes via the local structure of vertices. Pre-complexes provide a convenient framework for this. 

\begin{definition}\label{def: Asterisk}
  An \emph{asterisk} is a pre-graph with one vertex, such that $E=E^{\iota}\sqcup  E^{\tau}$.  If the 1-skeleton of a pre-complex $X$ is an asterisk then $X$ is called an \emph{asterisk pre-complex}. 
\end{definition}
  
Every vertex of a pre-graph or pre-complex now has a regular neighbourhood, based on an asterisk.

 \begin{definition}\label{def: Regular neighbourhood of a vertex}
 Let $G$ be a pre-graph and $v$ be a vertex of $G$.  The \emph{regular neighbourhood} of $v$ is an asterisk $N_G(v)$, with vertex set $\{v\}$ and edge set $\iota^{-1}(v)\sqcup \tau^{-1}(v)$. (Note that the disjoint union here is important: an edge $e$ with $\iota(e)=\tau(e)=v$ defines two edges in $N_G(v)$.)  The attaching maps are the unique maps $\iota^{-1}(v)\to \{v\} $ and  $\tau^{-1}(v)\to \{v\} $
  
Now let $X=(G,S,w)$ be a pre-complex and $v$ be a vertex of $X$.  The \emph{regular neighbourhood} $N=N_X(v)$ of $v$ is a pre-complex with 1-skeleton $G_N=N_G(v)$, i.e.\ the regular neighbourhood of $v$ in the 1-skeleton $G$. The faces of $N_X(v)$ are pulled back from $X$, so $S_N=S\times_G G_N$, and the attaching map $w_N$ is the pullback of $w$, i.e.\ the natural projection $S_N\to G_N$.
 \end{definition}
 
The regular neighbourhood of a vertex $v$ carries the same information as the \emph{link} or \emph{Whitehead graph} of the vertex. Indeed, any asterisk pre-complex defines a Whitehead graph.

\begin{definition}\label{def: Whitehead graph}
Let $X=(G,S,w)$ be an asterisk pre-complex with unique vertex $v$. The corresponding \emph{Whitehead graph} has vertex-set equal to the edges of $X$ and edge-set equal to $w^{-1}(v)$. Each $u\in w^{-1}(v)$ is incident at exactly two edges $\epsilon_+(v)$ and $\epsilon_-(v)$ of $S$, and the attaching maps of the Whitehead graph adjoin $u$ to $w(\epsilon_+(v))$ and $w(\epsilon_-(v))$.
\end{definition}

For $v$ a vertex of a complex $X$, the Whitehead graph of the regular neighbourhood $N(v)$ coincides with the Whitehead graph $\Wh(v)$ of \cite{wilton_essential_2018}. This in turn generalises the graph used by Whitehead in his solution to the orbit problem for automorphism groups of free groups \cite{whitehead_equivalent_1936}.  See also \cite{manning_virtually_2010,cashen_line_2011} for related recent uses of Whitehead graphs.

Edges of pre-complexes also have regular neighbourhoods.

\begin{definition}\label{def: Regular neighbourhoods of edges}
Let $X=(G,S,w)$ be a pre-complex and $e$ be an edge of $X$. The \emph{regular neighbourhood} of $e$ is a pre-complex $N=N_X(e)$ with 1-skeleton $G_N$ equal to the pre-graph that has no vertices and a single edge $\{e\}$. The faces $S_N$ of $N_X(e)$ consists of the pre-image $w^{-1}(e)$, and the attaching map $w$ is the restriction of $w$ to $S_N$.
\end{definition}
 
\begin{remark}\label{rem: Induced morphisms}
A morphism of pre-complexes or pre-graphs induces morphisms on regular neighbourhoods: $f_*:N(x)\to N(f(x))$, for $x$ a vertex or edge of the domain.
\end{remark}

\subsection{Various kinds of maps}

Using regular neighbourhoods, we can define the various kinds of maps that will play a role in the argument.

\begin{definition}\label{def: Branched maps and branched immersions}
A morphism of pre-complexes $f:X\to Y$ is  a \emph{branched map} if the induced map
\[
f_*:N_X(e)\to N_X(f(e))
\]
is injective for every edge $e$ of $X$. If the induced map
\[
f_*:N_X(v)\to N_X(f(v))
\]
is injective for every vertex $v$ of $X$ then $f$ is a \emph{branched immersion}.
\end{definition} 

Note that the realisation of a branched map of complexes $\real{f}:\real{X}\to\real{Y}$ is locally injective everywhere except possibly at the centres of the faces and the vertices, and a branched immersion is locally injective everywhere except possibly at the centres of the faces.

The number of times that a morphism $f$ winds a face around its image is an important quantity, called the \emph{degree}.

\begin{definition}\label{def: Degree and immersions}
Let $f:X\to Y$ be a morphism of 2-complexes. Since $f:S_X\to S_Y$ is a local immersion it is a covering map; in particular, for every path-component $C\subseteq S_X$, the restriction $f:C\to S_Y$ has a well-defined \emph{degree} $\deg_C(f)\in\NN$.  The sum
\[
\deg(f):=\sum_{C\in\pi_0(S_X)}\deg_C(f)
\]
is the \emph{degree of $f$}. If $f$ is a branched immersion and $\deg_C(f)=1$ for every component $C$ then $f$ is called an \emph{immersion}.
\end{definition}

\begin{remark}\label{rem: Immersions in orbicomplexes vs complexes}
In \cite{louder_one-relator_2020}, immersions were defined in the more general context of \emph{orbi}complexes, which allow cone points in the centres of 2-cells. In this paper, we only consider genuine 2-complexes, and the notion of immersions from \cite{louder_one-relator_2020} specialises exactly to the definition given in Definition \ref{def: Degree and immersions}.
\end{remark}

Note that a morphism $f$ is an immersion if and only if its
realisation $\real{f}$ is locally injective.  Sometimes we need to be
able to modify the 1-skeleton of a 2-complex without changing the
faces. The next kind of map makes this possible.

\begin{definition}\label{def: Face equivalence}
A morphism of finite pre-complexes $f:Y\to X$ is called a \emph{face-embedding} if the map of pre-graphs $f:S_Y\to S_X$ is injective, and a \emph{face-equivalence} if it is a bijection.
\end{definition}

\begin{remark}\label{rem: Face embedding}
If $f\colon X\to Y$ is a face-embedding of complexes then the restriction of $\boldsymbol{f}$ to the interiors of the two-cells of $\boldsymbol{X}$ is a homeomorphism onto its image, whence the term. Furthermore, if $f\colon X\to Y$ and $g\colon Y\to Z$ are essential face embeddings, then the composition $g\circ f\colon X\to Z$ is an essential face embedding.
\end{remark}

Stallings famously popularised the folding operation for graphs \cite{stallings-folding}.  In \cite{louder-wilton2,louder_one-relator_2020}, we made use of a folding operation on morphisms of 2-complexes that produces immersions.  Here we will make use of a very natural folding operation on morphisms of 2-complexes that produces branched immersions.

\begin{definition}\label{def: Folding 2-complexes}
Let $f:Y\to X$ be a morphism of pre-complexes.  As in \cite[\S3.3]{stallings-folding}, the induced map of 1-skeleta $f:G_Y\to G_X$ factors as
\[
G_Y\stackrel{f_0}\to G_{\bar{Y}}\stackrel{f_1}{\to} G_X
\]
where $f_0$ is a $\pi_1$-surjective finite composition of folds and $f_1$ is an immersion. The map of faces $f:S_Y\to S_X$ factors through the fiber product $G_{\bar{Y}}\times_{G_X} S_X$ by the universal property; set $S_{\bar{Y}}$ equal to its image, and let $w_{\bar{Y}}:S_{\bar{Y}}\to G_{\bar{Y}}$ be the natural projection. The data $(G_{\bar{Y}},S_{\bar{Y}},w_{\bar{Y}})$ define a pre-complex ${\bar{Y}}$ and $f$ factors as
\[
Y \stackrel{f_0}{\to} {\bar{Y}}\stackrel{f_1}{\to} X
\]
where $f_0$ is a $\pi_1$-surjection and $f_1$ is a branched immersion. The pre-complex ${\bar{Y}}$ is called the \emph{folded representative} of $Y$, and the map $f_1$ is called the \emph{folded representative} of $f$. {The folded representative is characterised by its universal property: if $f$ factors through a branched immersion $Z\to X$ then the folded representative $f_1$ also factors uniquely through $Z\to X$.}
\end{definition}

This enables us to define the last kind of map that we will need.

\begin{definition}\label{def: Face-essential map}
A morphism of finite pre-complexes $f:Y\to X$  is called \emph{face essential} if the map to the folded representative $f_0:Y\to\bar{Y}$ is a face-equivalence. If $X$ and $Y$ are both complexes, this is equivalent to requiring that $\deg(f)=\deg(f_1)$. {If, in addition, the restriction of $f$ to each connected component of $Y$ induces a $\pi_1$-injective map on 1-skeleta, then $f$ is said to be \emph{essential}.}
\end{definition}

Intuitively, a face-essential map is a map such that folding doesn't identify any faces.

\begin{remark}\label{rem: Boundary essential maps and Euler characteristic}
If $Y\to X$ is a face-essential morphism of finite complexes and $\bar{Y}$ is the folded representative of $Y$ then $\chi(\bar{Y})\geq\chi(Y)$, with equality if and only $G_Y\to G_{\bar{Y}}$ is a homotopy equivalence.
\end{remark} 

We will be interested in local ways of certifying that a morphism $f$ is face-essential. This turns out to be easier to do in a two-step process, encapsulated by the notion of a \emph{face immersion}.

\begin{definition}\label{def: Face immersion}
A \emph{face immersion} for a map $f:Y\to X$ is a factorisation
\[
Y\stackrel{f_0}{\to} Z\stackrel{f_1}{\to} X
\]
where $f_0$ is a face-equivalence and $f_1$ is a branched immersion.
\end{definition}

For brevity of notation, we will often abuse notation and use the map $f$ to denote a face immersion for $f$. A face immersion certifies that $f$ is face-essential, because of the following lemma.

\begin{lemma}\label{lem: Face-essential maps and face immersions}
Let $f:Y\to X$ be a morphism of finite 2-complexes.  There is a face immersion for $f$ if and only if $f$ is face-essential.
\end{lemma}
\begin{proof}
If $f$ is face essential then the factorisation through the folded representative $\bar{Y}$ defines a face immersion. Conversely, the existence of a face immersion implies that $f$ factors as
\[
Y\to\bar{Y}\to Z\to X
\]
by the universal property of the folded representative. Since $S_Y\to S_{\bar{Y}}$ is surjective { and the bijection $S_Y\to S_Z$ factors through it,} it is also bijective, so $Y\to\bar{Y}$ is { a face equivalence} and hence { according to Definition~\ref{def: Face-essential map} $f$ is face essential}.
\end{proof}

\subsection{Gluing pre-complexes}

One advantage of working with pre-complexes is that they can be conveniently glued to create complexes.

\begin{definition}\label{def: Gluing}
Let $f:X\to Y$ be a branched map of pre-complexes. Suppose that $e_1,e_2$ are edges of $G_X$ such that $f(e_1)=f(e_2)$.  Since $f$ is a branched map, it induces injective maps
\[
f_*:w_X^{-1}(e_1)\to w_Y^{-1}(f(e_1))
\]
and
\[
f_*:w_X^{-1}(e_2)\to w_Y^{-1}(f(e_2))\,.
\]
The edges $e_1$ and $e_2$ of $X$ are said to be \emph{gluable (along $f$)} if
\begin{enumerate}[(i)]
\item $e_1\notin E^\iota_X$ and $e_2\notin E^\tau_X$, and
\item $f_*(w_X^{-1}(e_1))=f_*(w_X^{-1}(e_2))$.
\end{enumerate}
In this case, $f$ induces a well-defined bijection $w_X^{-1}(e_1)\to w_X^{-1}(e_2)$.

We can now define a new pre-complex $X'$ from $X$ by identifying $e_1$ and $e_2$ to a common edge $e$. Set $\iota(e)=\iota(e_2)$ (if it exists) and $\tau(e)=\tau(e_1)$ (if it exists). Each edge $w_X^{-1}(e_1)$  corresponds to a unique edge $\epsilon_2\in w_X^{-1}(e_2)$ under the above bijection, and in $X'$ these two edges are identified to a common edge $\epsilon\in w_{X'}^{-1}(e)$, with $\iota(\epsilon)=\iota(\epsilon_2)$ (if it exists) and $\tau(\epsilon)=\tau(\epsilon_1)$ (if it exists). Finally, note that $f$ descends to a well-defined branched map $f:X'\to Y$.

The resulting pre-complex $X'$ is said to be constructed from $X$ by \emph{gluing $e_1$ and $e_2$ (along $f$)}.
\end{definition}

Gluing can be used to reassemble a complex from the regular neighbourhoods of its vertices.

\begin{remark}\label{rem: Gluing regular neighbourhoods}
Let $X$ be a finite complex, and let $N$ be the disjoint union of the regular neighbourhoods of the vertices of $X$. Then $N$ is naturally equipped with a branched immersion $N\to X$, and $X$ can be reconstructed from $N$ by gluing the edges in pairs.
\end{remark}

\section{Irreducible complexes}\label{sec: Irreducible complexes}

The properties of non-positive immersions, negative immersions and uniform negative immersions have a common theme. Given an immersion of finite 2-complexes $Y\to X$, if the Euler characteristic $\chi(Y)$ is { not negative enough,}  then $Y$ can be simplified.  The exact sense in which $Y$ can be simplified is of crucial importance to the definition of uniform negative immersions. In the next subsection we define irreducible complexes, and with these definitions in hand we can give the correct definitions of non-positive, negative and uniformly negative immersions.

\subsection{Unfolding}

We are interested in ways that a finite 2-dimensional complex $X$ might be homotopic to a point. The next definition summarises various ways in which $X$ either is a point or can be simplified.

\begin{definition}[Visibly reducible and irreducible pre-complexes]\label{def: Visibly reducible and visibly irreducible}
A finite pre-complex $X$ is \emph{visibly reducible} if any one of the following four conditions hold:
\begin{enumerate}[(i)]
\item some component of $X$ is a point;
\item the 1-skeleton $G$ has a vertex of valence 1;
\item $X$ has a \emph{free face} --- that is, an edge $e$ of the 1-skeleton $G$ with $\#w^{-1}(e)=1$;
\item a vertex $v$ of $X$ separates its regular neighbourhood $N_X(v)$, so $N_X(v)\setminus\{v\}$ is disconnected.
\end{enumerate}
A fifth kind of reduction also plays an important role:
\begin{enumerate}[(i)]
\setcounter{enumi}{4}
\item  a vertex $v$ of $X$ has an incident edge $e$ such that the union $v\cup e$ separates the regular neighbourhood $N_X(v)$.
\end{enumerate}
If none of conditions (i)-(v) hold then $X$ is called \emph{visibly irreducible}.  If condition (v) holds but none of conditions (i)-(iv) hold, then $X$ is called \emph{unfold-able}\footnote{The hyphen is to resolve an ambiguity: we mean that such a pre-complex can be unfolded, not that it cannot be folded.}. { See Figure~\ref{fig: Torus folding} for a concrete example of an unfold-able 2-complex.}
\end{definition}

{If neither of conditions (i) or (ii) hold then the 1-skeleton of $X$ is a core graph. Recall that a graph is called \emph{core} if every vertex is in the image of an immersed cycle. Henceforth, we will always implicitly assume that the 1-skeleton of any 2-complex we consider is a core graph.}

{The conditions of Definition \ref{def: Visibly reducible and visibly irreducible}} can also be phrased naturally in terms of Whitehead graphs: see for instance \cite[Lemma 2.10]{wilton_essential_2018} for a similar list of conditions. Recall from \cite{stallings-folding} that a morphism of graphs $f:G\to H$ is a \emph{fold} if it identifies just a single pair of edges $e_1,e_2$, incident at a common vertex. 

\begin{definition}\label{def: Essential fold and essential equivalence}
A fold of graphs is called \emph{essential} if it is injective on fundamental groups.  A morphism of graphs $f:G\to H$ is called an \emph{essential equivalence} if it is a composition of essential folds, or equivalently, it is an isomorphism on fundamental groups. A face-essential morphism of complexes $f:Y\to X$ is called an \emph{essential equivalence} if the map of 1-skeleta $f:G_Y\to G_X$ is an essential equivalence of graphs.   In this case, we also say that $Y$ is an \emph{essential unfolding} of $X$.
\end{definition}

Note that, if $f:Y\to X$ is an essential equivalence of complexes, then the realisation $\real{f}$ is a homotopy equivalence.

The next lemma, which is a restatement of \cite[Lemma 2.8]{wilton_essential_2018} in the terminology of this paper, explains the relationship between Definition \ref{def: Visibly reducible and visibly irreducible}(v) and (un)folding.

\begin{lemma}\label{lem: Unfolding one step}
If a finite complex $X$ is unfold-able then there is a finite complex $X'$ and an essential fold $f:X'\to X$.
\end{lemma}

Unfolding preserves the {number of edges of the faces}  while increasing the number of vertices and edges of the 1-skeleton, so it is not hard to see that Lemma \ref{lem: Unfolding one step} can only be applied finitely many times before producing a complex which is either visibly reducible or visibly irreducible. The next lemma is a restatement of \cite[Lemma 2.11]{wilton_essential_2018} in the terminology of this paper.

\begin{lemma}\label{lem: Unfolding all the way}
For any finite complex $X$, there is a finite complex $X'$ and an essential equivalence $f:X'\to X$ such that $X'$ is either visibly reducible or visibly irreducible.
\end{lemma}

We say that $X$ \emph{unfolds to} $X'$. To summarise, unfold-able complexes may be either reducible or irreducible, but their status can be determined by unfolding.

\begin{definition}[Reducibility and irreducibility]\label{def: Reducible and irreducible complexes}
If a finite 2-complex $X$ unfolds to a visibly reducible complex $X'$ {(with the 1-skeleton still a core graph)} then $X$ is said to be \emph{reducible}. Likewise, if $X$ unfolds to a visibly irreducible complex $X'$ then $X$ is said to be \emph{irreducible}.
\end{definition}

Note that an irreducible 2-complex is not required to be connected.

\begin{remark}\label{rem: Terminology in other papers} 
The terminology `irreducible' has been used in several papers related to this one, such as \cite{louder-wilton,louder-wilton2,louder_one-relator_2020,wilton_essential_2018}. Although the usage was always similar to its definition given here, it was not always identical, so the reader is advised to check the definition in the cited paper when a result is quoted.
\end{remark}

It is not immediately obvious that a complex $X$ cannot be both reducible and irreducible, since $X$ may admit many different unfoldings.  Indeed, a 2-complex may have several unfoldings, as the next example illustrates.

\begin{figure}[ht]
  \centerline{
    \includegraphics[width=\textwidth]{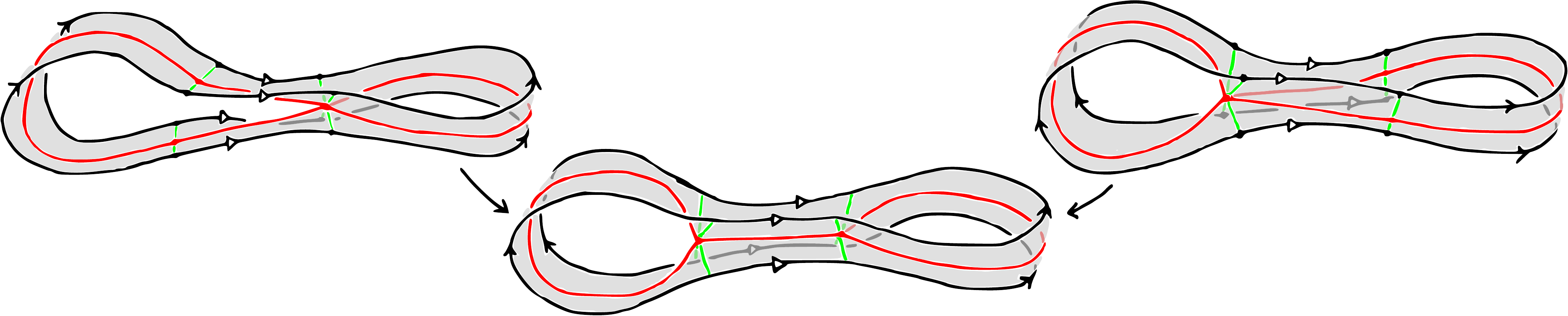}
  }
  \caption{The central diagram illustrates the complex $X$ from Example \ref{eg: Folding a torus}: the 1-skeleton is coloured red, and the single face is attached along the black curve. It admits two different unfoldings, each a realisation of a torus.}
  \label{fig: Torus folding}
\end{figure}

\begin{example}\label{eg: Folding a torus}
Let $X$ be a 2-complex defined as follows.  The 1-skeleton of $X$ is a `spectacles' or `bar-bell' graph, with two vertices $u_1$ and $u_2$, and edges $a_1,a_2,b$, where each $a_i$ is a loop starting and ending at $u_i$, and $b$ joins $u_1$ to $u_2$. The single face $C$ of $X$ is attached along the loop labelled by the commutator of $a_1$ and $ba_2b^{-1}$.  Now $X$ admits two different unfoldings, each with realisation homeomorphic to the torus: one is obtained by dividing $u_1$ into two vertices, the other is obtained by dividing $u_2$ into two vertices.
\end{example}

Many further examples of this type can be constructed by taking any cellular decomposition $X'$ of a surface and performing an essential fold to obtain a complex $X$; such an $X$ can always be unfolded in two ways.

Nevertheless, a famous lemma of Whitehead \cite{whitehead_equivalent_1936} implies that (ir)reducibility is equivalent to various group-theoretic properties of the conjugacy classes in the free group $F=\pi_1(G)$ represented by the components of $S$.  Combined with the material of this section, Whitehead's lemma (see, for instance, \cite[Lemma 2.10]{wilton_essential_2018} for a suitable statement), gives the following proposition. 

\begin{proposition}\label{prop: Whitehead}
Let $X$ be a finite connected complex. Let $F$ be the fundamental group of the 1-skeleton $G_X$, and let $\{w_1,\ldots,w_n\}$ be a finite set of representatives for the conjugacy classes of $F$ determined by the components of $S$. The complex $X$ is reducible if and only if one of the following three conditions hold:
\begin{enumerate}[(i)]
\item $F$ is trivial and $n=0$;
\item $F$ is infinite cyclic, {$n\leq 1$ and, if $n=1$,} $F=\langle w_1\rangle$;
\item there is a non-trivial free splitting $F=A*B$ such that, for each $i$, $w_i$ is conjugate into either $A$ or $B$.
\end{enumerate}
\end{proposition}

Since the conditions of Proposition \ref{prop: Whitehead} are group-theoretic and are not altered by homotopy-equivalences of the 1-skeleta (once we have ensured that they are core graphs), it follows that every finite complex is either reducible or irreducible.

\subsection{Irreducible cores}

Since this notion of irreducibility is slightly stronger than those used in \cite{louder-wilton,louder-wilton2}, we will need some lemmas that enable us to improve unfold-able complexes to irreducible ones, under suitable hypotheses on the fundamental group.

\begin{definition}\label{def: Freely decomposable}
A group $\Gamma$ is \emph{freely decomposable} if $\Gamma$ splits as a non-trivial free product $\Gamma=A*B$. Otherwise $\Gamma$ is said to be \emph{freely indecomposable}.
\end{definition}

A 2-complex $X$ can be replaced by a visibly irreducible one, as long as the fundamental group is complicated enough.

\begin{lemma}\label{lem: Irreducible core}
Let $X$ be a finite, connected 2-complex, and suppose that $\pi_1(X)$ is neither freely decomposable nor free. Then there is a visibly irreducible 2-complex $X'$ and an essential face-embedding $X'\to X$, which is also an isomorphism on fundamental groups.
\end{lemma}
\begin{proof}
The proof proceeds by induction on the number of faces of $X$. By Lemma \ref{lem: Unfolding all the way}, there is an essential equivalence $X'\to X$ such that $X'$ is either visibly reducible or visibly irreducible. If $X'$ is visibly irreducible the result is immediate; otherwise, $X'$ is visibly reducible. Since $\pi_1(X)\cong\pi_1(X')$ is non-trivial, $X'$ is not a point, and we may assume that the 1-skeleton of $X'$ is a core graph. 

If $X'$ has a free face, collapsing it exhibits a sub-complex $X''\subseteq X'$ which is a deformation retract of $X'$ with fewer faces. Note also that the inclusion of $X''$ into $X'$ is an essential face-embedding. Therefore, the result holds for $X''$, and hence for $X'$,  by the inductive hypothesis.

Finally, suppose that $X'$ has a locally separating vertex $v$. Since $\pi_1(X')$ is freely indecomposable and not infinite cyclic, $v$ must be separating, and cutting $X'$ along $v$ realises $X'$ as a wedge
\[
X'=X_1\vee X_2\,.
\]
{Note that the inclusion maps $X_i\into X'$ are essential.} Since every free splitting of $\pi_1(X')$ is trivial, $X_2$ is simply connected, without loss of generality. Since the 1-skeleton of $X'$ is a core graph, $X_2$ is not a tree, and must therefore include at least one face. Hence $X_1$ has fewer faces than $X$, and so the lemma holds for $X_1$ by the inductive hypothesis. Since $X_1\to X'$ is an isomorphism on fundamental groups, the result also holds for $X$, which completes the proof.
\end{proof}

As in Example \ref{eg: Folding a torus}, the complex $X'$ furnished by Lemma \ref{lem: Irreducible core} may not be unique. However, we can make it unique by folding.

\begin{definition}\label{def: Irreducible core}
Let $X$ be finite, connected 2-complex with $\pi_1(X)$ neither freely decomposable nor free.  Let $f:X'\to X$ be provided by Lemma \ref{lem: Irreducible core}, and let $\widehat{X}\to X$ be the immersion provided by folding $f$. Then $\widehat{X}$ is called an \emph{irreducible core} for $X$.
\end{definition}

\begin{remark}\label{rem: Remarks about irreducible cores}
{Since $X'\to X$ is essential,} $X'$ is an essential unfolding of $\widehat{X}$. In particular, since $X'$ is visibly irreducible, it follows that $\widehat{X}$ is irreducible.  Also, since $X'\to X$ is a $\pi_1$-isomorphism and factors through the $\pi_1$-surjection $X'\to\widehat{X}$, the immersion $\widehat{X}\to X$ is a $\pi_1$-isomorphism.
\end{remark}

We can also characterise irreducible cores group-theoretically.  Let $X=(G_X,S_X,w_X)$ be a compact, connected 2-complex. The fundamental group of the 1-skeleton $G_X$ is a finitely generated free group $F$. The images of the circles $S_X$ under the attaching map $w_X$ define a collection of conjugacy classes of cyclic subgroups $\{\langle w_1\rangle,\ldots, \langle w_n\rangle \}$ of $F$. Grushko's theorem implies the existence of a maximal splitting
\[
F= F_1*\ldots* F_k\,,
\]
such that each $w_i$ is conjugate into some $F_j$. (If we combine the $F_j$ that do not contain any of the $w_i$ into a common free factor $H$, then this gives the \emph{relative Grushko decomposition} of the pair $(F,\{\langle w_i\rangle\})$.) If $\Gamma_j$ is the quotient of $F_j$ by the conjugates of the $w_i$ contained in $F_j$, then this gives a natural free-product decomposition
\[
\Gamma=\Gamma_1 *\ldots * \Gamma_k
\]
of $\Gamma=\pi_1(X)$. If $\Gamma$ is neither freely decomposable nor free, then exactly one of these factors -- without loss of generality, $\Gamma_1$ -- must be non-trivial. Now, $\widehat{X}$ can be defined as follows: the 1-skeleton is the Stallings core of the subgroup $F_1$ of $F=\pi_1(G_X)$, and the faces correspond to those $w_i$ that are conjugate into $F_1$.

We shall see that irreducible cores enjoy a universal property. Before we prove that, we need to notice that unfolds can be pulled back along branched immersions. 

\begin{lemma}\label{lem: Unfoldings pull back along branched immersions}
Let $Y\to X$ be a branched immersion, and suppose that $X'\to X$ is an essential fold. Then there is an essential unfolding $Y'\to Y$ such that the composition $Y'\to X$ factors through $X'\to X$.
\end{lemma}
\begin{proof}
Such an unfolding of $X$ means that there is an edge $e$ with endpoints $u$ and $v$ such that $e$ and $v$ together separate the regular neighbourhood $N_X(v)$, so it can be written as a union
\[
N_X(v)= N_1\cup_{e\cup v} N_2\,.
\]
Let $v_1,\ldots, v_n$ be the pre-images of $v$ in $Y$ and let $e_i$ be the pre-images of $e$ incident at $v_i$, if it exists. For each $i$, $N_Y(v_i)$ embeds in $N_X(v)$, and so either  $N_Y(v_i)$ maps into one of $N_1$ or $N_2$, or $N_Y(v_i)$ is separated by $e_i\cup v_i$.  The unfolding $Y'$ is now constructed by unfolding each of these separating $e_i$ {in turn}.
\end{proof} 

Using this, we can now prove the claimed universal property for irreducible cores.

\begin{lemma}\label{lem: Universal property of irreducible cores}
Let $Y\to X$ be a branched immersion of finite, connected 2-complexes that is non-trivial on fundamental groups. Suppose that $Y$ is irreducible and that $\pi_1(X)$ is neither freely decomposable nor free, and let $\widehat{X}\to X$ be an irreducible core of $X$. Then the morphism $Y\to X$ factors through the map $\widehat{X}\to X$.
\end{lemma}
\begin{proof}
Recall that the irreducible core $\widehat{X}$ is constructed by folding an inductively constructed morphism $X'\to X$.  The lemma is proved by arguing that there is an essential unfolding $Y'\to Y$ that lifts to $X'$. 

Consider the inductive construction of $X'$ from Lemma \ref{lem: Irreducible core}, which is a finite sequence of unfoldings and inclusions of sub-complexes.  Lemma \ref{lem: Unfoldings pull back along branched immersions} handles the case of unfoldings. It remains to deal with the two cases of the argument that pass to subcomplexes of an unfolding $X'$ of $X$. 

In the first case, $X'$ has a face $C$ that is a free face, and $X''\subseteq X'$ is the subcomplex obtained by collapsing that free face.  Since $Y'\to X'$ is a branched immersion, if any face $D$ of $Y'$ maps to $C$ then $D$ is also a free face, so $Y'$ is visibly {reducible}, contradicting the hypothesis that $Y$ is irreducible. Therefore $C$ is not in the image of $Y'$, and so the map $Y'\to X'$ restricts to $X''$.

In the second case, $X'$ has a locally separating vertex and splits as a wedge $X_1\vee X_2$ with $X_2$ simply connected. Since $Y'\to X'$ is a branched immersion and no vertex of $Y'$ is locally separating, it follows that the image of $Y$ is contained entirely inside either $X_1$ or $X_2$, and since $Y'\to X'$ is non-trivial on fundamental groups, it must be $X_1$, as required.  This completes the proof that $Y$ unfolds to some $Y'$ such that $Y'\to X$ lifts to $X'$.

To complete the proof, we fold $Y'$ to construct a 2-complex $\widehat{Y}$.  Since $\widehat{X}\to X$ and $Y\to X$ are branched immersions, there is a well-defined fibre-product 2-complex $Y\times_X\widehat{X}$, with 1-skeleton equal to the fibre product of graphs ${G_Y}\times_{G_X}G_{\widehat{X}}$ and faces given by pulling back the faces of $Y$ and $\widehat{X}$.  The universal property of the fibre product now provides a canonical map $Y'\to Y\times_X\widehat{X}$, and $\widehat{Y}$ is the result of folding this map. By construction, the natural map $\widehat{Y}\to X$ factors through $\widehat{X}$.

Finally, since the identity map on $Y$ is the folded representative of the map $Y'\to Y$, the universal property of folded representatives provides a canonical morphism $Y\to\widehat{Y}$, and since $\widehat{Y}\to X$ factors through $\widehat{X}$, the result follows.
\end{proof}

\begin{remark}\label{rem: Irreducible cores are well defined}
The universal property of Lemma \ref{lem: Universal property of irreducible cores} implies in the usual way that irreducible cores are unique up to canonical isomorphism.  We may therefore speak of \emph{the} irreducible core of a finite, connected 2-complex $X$, as long as $\pi_1(X)$ is neither freely decomposable nor free.
\end{remark}

\begin{remark}\label{rem: Irreducible cores via relative Grushko}
The irreducible core of a connected 2-complex $X$ can also be characterised group-theoretically. Let $F$ be the fundamental group of the 1-skeleton, and let $\langle w_1\rangle,\ldots,\langle w_n\rangle$ be the cyclic subgroups of $F$ defined (up to conjugacy) by the faces of $X$. The pair $(F,\{\langle w_1\rangle,\ldots,\langle w_m\rangle\})$ has a relative Grushko decomposition $F=F_1*\ldots *F_n$.  Each $w_i$ is conjugate into a unique $F_{j(i)}$ and so, after conjugating the $w_i$, the pair $(F,\{\langle w_i\rangle\})$ decomposes as a free product of pairs $(F_k,\{\langle w_i\rangle\mid j(i)=k\})$. If $\pi_1(X)$ is neither freely decomposable nor free, then there is a unique index $k$ such that $F_k/\ncl{w_i\mid j(i)=k}$ is neither freely decomposable nor free. The 1-skeleton of the irreducible core $\widehat{X}$ is given by the core graph associated to the subgroup $F_k$ of $F$, and the 2-cells are given by the set $\{w_i\mid j(i)=k\}$.
\end{remark}

The results of this section guarantee that $X$ has an irreducible core if it is neither freely decomposable nor free. However, the converse is not true: $X$ may be irreducible, and yet $\pi_1(X)$ may be trivial, free,  freely indecomposable etc. For instance, the complex associated to the presentation
\[
\langle a,b\mid baba^{-2}, abab^{-2}\rangle
\]
of the trivial group is irreducible.

\subsection{Curvature conditions}

With the definitions of irreducibility in hand, we can now define non-positive and negative immersions.   These definitions are inspired by Wise \cite{wise_coherence_2003,wise_coherence_2020}, and similar definitions have played a role in several recent papers about one-relator groups, such as \cite{helfer-wise,louder-wilton,louder-wilton2}.  They are also naturally related to the work of Mart\'inez-Pedroza and Wise on sectional curvature for 2-complexes \cite{wise_sectional_2004,wise_nonpositive_2008,martinez-pedroza_coherence_2013}. 

It turns out that Wise's original definitions are too strong to prove Theorem \ref{thm: Coherence}. Here we will give slightly less restrictive definitions, using the notion of irreducible complexes from the last section. See Section \ref{sec: Wise's notion of sectional curvature} below for a comparison between our definitions and Wise's. As usual, $\chi(X)$ denotes the Euler characteristic of a complex $X$.

\begin{definition}\label{def: Non-positive immersions}
A finite 2-complex $X$ has \emph{non-positive immersions} if, for every immersion of finite complexes $Y\immerses X$, either
\begin{enumerate}[(i)]
\item $\chi(Y)\leq 0$, or
\item $Y$ is reducible.
\end{enumerate}
\end{definition}

{The non-positive immersions property} for presentation complexes of one-relator groups was proved in \cite{helfer-wise} and \cite{louder-wilton}.

The next definition is motivated by the observation that, if $f:Y\to X$ is an immersion, then $\deg(f)$ is equal to the number of faces of $Y$, and can be thought of as a useful adaptation of Euler characteristic to the setting of branched maps.  

\begin{definition}\label{def: Total curvature}
Let $f:Y\to X$ be a branched map of 2-complexes, and let $G$ be the 1-skeleton of $Y$. The quantity
\[
\tau(f):=\deg(f)+\chi(G)
\]
is called the \emph{total curvature} of $f$. We will usually abuse notation and write $\tau(Y)$ for $\tau(f)$, since  $\tau(f)=\chi(Y)$ when $f$ is an immersion.
\end{definition}

Indeed, the techniques of both papers \cite{helfer-wise,louder-wilton} prove something slightly stronger than non-positive immersions: the map can be taken to be a branched immersion, and the complex $Y$ is only required to not be visibly reducible.    For instance, \cite[Theorem 1.2]{louder-wilton} can be stated as follows, after accounting for small differences in terminology.

\begin{theorem}\label{thm: NPI for one-relator groups}
Let $X$ be the presentation complex of a torsion-free one-relator group.  If $Y$ is a finite complex and $Y\to X$ is a branched immersion then either
\begin{enumerate}[(i)]
\item $\tau(Y)\leq 0$, or
\item $Y$ has a free face, and in particular is visibly reducible.
\end{enumerate}
\end{theorem}

The definition of negative immersions is a natural adaptation of Definition \ref{def: Non-positive immersions}.  Again, this definition assumes a slightly stronger notion of irreducibility than was used in \cite{louder-wilton2}.

\begin{definition}\label{def: Negative immersions}
A finite 2-complex $X$ has \emph{negative immersions} if, for every immersion of finite complexes $Y\to X$,  either
\begin{enumerate}[(i)]
\item $\chi(Y)< 0$, or
\item $Y$ is reducible.
\end{enumerate}
\end{definition}

The property of negative immersions for one-relator groups was characterised in \cite{louder-wilton2}, {in which Theorem \ref{thm: Negative immersions} was proved.}

\begin{remark}\label{rem: Definitions of NI are consistent}
In \cite[Definition 1.1]{louder-wilton2}, a 2-complex $X$ is defined to have negative immersions if, for every immersion of finite complexes $Y\to X$, either $\chi(Y)<0$ or $Y$ Nielsen reduces to a graph. To say that $Y$ Nielsen reduces to a graph means that there is a homotopy equivalence of 1-skeleta $Y_{(1)}\simeq Z_{(1)}$ so that the 2-complex $Z$ obtained by attaching the 2-cells of $Y$ to $Z_{(1)}$ in the natural way is the result of wedging discs to a graph.

Let us briefly explain why Definition \ref{def: Negative immersions} agrees with \cite[Definition 1.1]{louder-wilton2}. To this end, let $Y\to X$ be an immersion of finite 2-complexes, and assume that $\chi(Y)\geq 0$. If $X$ satisfies \cite[Definition 1.1]{louder-wilton2} then $Y$ Nielsen reduces to a graph, and then $Y$ is reducible by Proposition \ref{prop: Whitehead}.

To prove the converse, assume that $X$ satisfies Definition \ref{def: Negative immersions}, and proceed by induction on the number of 2-cells of $Y$. Since $\chi(Y)\geq 0$, $Y$ is reducible by assumption, and Proposition \ref{prop: Whitehead} applies. Items (i) and (ii) of the proposition correspond to the base cases of the induction where $Y$ is a point, a circle or a disc. If item (iii) of Proposition \ref{prop: Whitehead} applies, then $Y$ unfolds to a 2-complex $Y'$ with a locally separating vertex $v$. Cutting along $v$ yields a (possibly disconnected) new 2-complex $Y'_1$, which folds to a 2-complex $Y_1$ immersing into $Y$ with $\chi(Y_1)=\chi(Y)+1$. Repeating this process $n$ times leads to a sequence of immersed 2-complexes
\[
Y_n\to Y_{n-1}\to\ldots\to Y_1\to Y
\]
with $\chi(Y_n)=\chi(Y)+n\geq n$, each with the same number of 2-cells as $Y$. Note that, for each $n$, $Y$ is Nielsen equivalent to the result of wedging $Y_n$ to a graph.

For large enough $n$, $Y_n$ must have a free face. Indeed, let $m$ be the number of 2-cells of $Y$ and let $l$ be their total edge length.  Then, unless $Y_n$ has a free face, it has at most $l/2$ edges and hence, since every vertex has valence at least 2, it has at most $l/2$ vertices, so $\chi(Y_n)\leq l/2+m$. Therefore, $Y_n$ must have a free face for $n>l/2+m$, as claimed. Collapsing a free face in such a $Y_n$ produces a new 2-complex $Y'$ with $\chi(Y')\geq 0$ and with fewer 2-cells than $Y$. Now $Y'$ Nielsen reduces to a graph by induction, and hence so does $Y$, as required.
\end{remark}

Theorem \ref{thm: Negative immersions} gives a useful statement, but the following more precise statement is an immediate consequence of \cite[Lemma 6.11]{louder-wilton2}, after noticing that the proof of that result works equally well if immersions are replaced by branched immersions and $\chi$ is replaced by $\tau$.

\begin{theorem}\label{thm: Precise negative immersions statement for branched immersions}
Let $X$ be the presentation complex of a torsion-free one-relator group $G=F/\ncl{w}$ with negative immersions. If $Y$ is a finite complex and $f:Y\to X$ is a branched immersion then either
\begin{enumerate}[(i)]
\item $\tau(Y)< 0$, or
\item $Y$ is reducible.
\end{enumerate}
\end{theorem}

Below, we will need the following improvement of this statement to the context of face-essential maps, 

\begin{corollary}\label{cor: Precise negative immersions statement for face-essential maps}
Let $X$ be the presentation complex of a torsion-free one-relator group $G=F/\ncl{w}$ with negative immersions. If $Y$ is a finite complex and $f:Y\to X$ is a face-essential map then either
\begin{enumerate}[(i)]
\item $\tau(Y)< 0$, or
\item $Y$ is reducible.
\end{enumerate}
\end{corollary}
\begin{proof}
Consider the folded representative
\[
Y\to \bar{Y}\to X
\]
so $\bar{Y}\to X$ is a branched immersion. Since $Y\to X$ is face essential,
\[
\tau(Y)\leq \tau(\bar{Y})\,,
\]
because $\deg{f}$ is equal to the degree of its folded representative, with equality if and only if $Y\to\bar{Y}$ is an essential equivalence.  Suppose now that $\tau(Y)=0$, so $\tau(\bar{Y})\geq 0$.  There are two cases to consider: either $\tau(\bar{Y})=0$ or $\tau(\bar{Y})>0$.

In the first case, $\tau(\bar{Y})=0$ so $Y\to\bar{Y}$ is an essential equivalence.  Theorem \ref{thm: Precise negative immersions statement for branched immersions} then implies that $\bar{Y}$ is reducible. Since reducibility is equivalent to the group-theoretic criteria of Proposition \ref{prop: Whitehead}, it follows that $Y$ is also reducible, because $Y\to\bar{Y}$ is an essential equivalence.

In the second case, $\tau(\bar{Y})>0$ so, by Theorem \ref{thm: NPI for one-relator groups}, $\bar{Y}$ has a free face, whence $Y$ also has a free face, and so is visibly reducible.
\end{proof}

Finally, we are ready to give a definition of \emph{uniform} negative immersions.

\begin{definition}\label{def: Uniform negative immersions}
A finite 2-complex $X$ has \emph{uniform negative immersions} if there is $\epsilon>0$ so that, for every immersion of finite complexes $Y\immerses X$, either
\begin{enumerate}[(i)]
\item \[
\frac{\chi(Y)}{\#\{2\mathrm{-cells\,of}\,Y\}}\leq -\epsilon\,,
\]
or
\item $Y$ is reducible.
\end{enumerate}
\end{definition}

Theorem \ref{thm: Uniform negative immersions} asserts that presentation complexes of one-relator groups with negative immersions also have uniform negative immersions.  

\subsection{Wise's notion of negative immersions}\label{sec: Wise's notion of sectional curvature}

We finish this section by contrasting Definition \ref{def: Uniform negative immersions} with the following definition of Wise.

\begin{definition}\label{def: Negative immersions in the sense of Wise}
A finite 2-complex $X$ has \emph{negative immersions in the sense of Wise} if there is $\epsilon>0$ so that, for every immersion of finite complexes $Y\immerses X$, either
\begin{enumerate}[(i)]
\item \[
\frac{\chi(Y)}{\#\{2\mathrm{-cells\,of}\,Y\}}\leq -\epsilon\,,
\]
or
\item $Y$ has a free face.
\end{enumerate}
\end{definition}

In \cite{wise_coherence_2020}, Wise proves that every one-relator group with torsion $G$ is coherent by showing that a presentation complex for (some finite-index subgroup of) $G$ satisfies Definition \ref{def: Negative immersions in the sense of Wise}, and then proving that this in turn implies coherence. The independent proof given in \cite{louder_one-relator_2020} can also be viewed this way.

Clearly, negative immersions in the sense of Wise implies our notions of negative immersions and uniform negative immersions. The following example shows that this implication is not reversible \cite{wise_coherence_2020}.

\begin{example}\label{eg: Negative immersions in the sense of Wise}
Let $X$ be the 2-complex associated to the presentation $G=F/\ncl{w}$, where $F$ is the free group $\langle a,b,c\rangle$ and
\[
w:=ab^2c^2ab^2c^2b^2c^2\,.
\]
By setting $a'=ab^2c^2$, this presentation can be seen to be Nielsen equivalent\footnote{Recall that two presentations are \emph{Nielsen equivalent} if one is taken to the other by applying a free-group automorphism to the generators.} to
\[
\langle a',b,c\mid (a')^2b^2c^2\rangle\,,
\]
the standard presentation of the fundamental group of the surface of Euler characteristic $-1$. In particular, $G$ is 2-free and so $X$ has negative immersions by Theorem \ref{thm: Negative immersions}, and indeed uniform negative immersions by Theorem \ref{thm: Uniform negative immersions}.

We now define a compact 2-complex $Y$ and an immersion to $X$.  The 1-skeleton of $Y$ is a rose with two petals labelled $\alpha$ and $\beta$, and immerses into $X$ via the assignment
\[
f:\alpha\mapsto a\,,\, \beta\mapsto b^2c^2\,,
\]
while a single face is attached along the word $u(\alpha,\beta)=\alpha\beta\alpha\beta^2$. Since $f(u(\alpha,\beta))=w(a,b,c)$, the immersion of 1-skeleta extends to an immersion $Y\immerses X$. Clearly $\chi(Y)=0$, but since both $\alpha$ and $\beta$ appear at least twice in $u$, $Y$ has no free face. Therefore, $Y$ does not have negative immersions in the sense of Wise.
\end{example}

The immersion $Y\immerses X$ does not contradict our definitions of (uniform) negative immersions (Definitions \ref{def: Negative immersions} and \ref{def: Uniform negative immersions}), since $u$ is a primitive element of the free group $\langle \alpha,\beta\rangle$, and so the complex $Y$ is reducible by Proposition \ref{prop: Whitehead}.

Example \ref{eg: Negative immersions in the sense of Wise} shows that one cannot prove Theorem \ref{thm: Coherence} by showing that all such complexes satisfy Definition \ref{def: Negative immersions in the sense of Wise}. It also illustrates another important difference between the definitions: since the standard presentation complex for the surface of Euler characteristic $-1$ does have negative immersions in the sense of Definition \ref{def: Negative immersions in the sense of Wise}, that notion is not invariant under Nielsen equivalence, whereas Definitions \ref{def: Negative immersions} and \ref{def: Uniform negative immersions} are, by Remark \ref{rem: Irreducible cores via relative Grushko}. 

\subsection{Stable commutator length}\label{sec: Stable commutator length}

Commutator length and its stabilisation provide natural quantifications of the complexity of an element of a commutator subgroup. See Calegari's monograph for a comprehensive treatment of stable commutator length \cite{calegari_scl_2009}.

\begin{definition}\label{def: (Stable) commutator length}
The \emph{commutator length} of an element $g$ of the commutator subgroup of a group $G$ is
\[
\cl(g)=\min\{n\in\NN \mid g = \prod_{i=1}^n [x_i,y_i]\,,\,x_i,y_i\in G\}\,,
\]
the minimal $n$ such that $g$ is a product of $n$ commutators. The \emph{stable commutator length} is defined to be
\[
\scl(g)=\inf_{n\in\NN}\frac{\cl(g^n)}{n}\,.
\]
\end{definition}

The purpose of this section is to prove Corollary \ref{cor: Weak Heuer}, and also to explain why Conjecture \ref{conj: Stability conjecture} implies Conjecture \ref{conj: Heuer's conjecture}. Both follow from the next proposition, which is due to Calegari \cite[Lemma 2.7]{calegari_surface_2008}.  Nevertheless, in order to clarify how the argument works in the setting of this paper, we give the proof here.

A 2-complex is called a \emph{surface} if its realisation is homeomorphic to a closed, 2-dimensional surface.

\begin{proposition}[Calegari]\label{prop: Essential map realising scl}
Let $X$ be the presentation complex of a one-relator group $F/\ncl{w}$ and suppose that $w\in [F,F]$. There is an essential branched map from an orientable surface $f:\widehat{\Sigma}\to X$ such that
\[
\frac{\tau(\widehat{\Sigma})}{\deg(f)}=1-2\scl(w)\,.
\]
\end{proposition}
\begin{proof}
As above, the data of $X$ consists of a map $w:S\immerses G$, where $G$ is the 1-skeleton and $w$ is the attaching map of the unique face.  An \emph{admissible surface} is a commutative diagram of continuous maps of topological spaces
\begin{center}
  \begin{tikzcd}
\partial\Sigma\arrow{r}\arrow{d} & \Sigma\arrow{d}\\
\real{S}\arrow{r}{w} &\real{G}
  \end{tikzcd}
\end{center}
such that $\Sigma$ is a compact surface without spherical components, $\partial\Sigma\to\Sigma$ is the natural inclusion of the boundary and $\partial\Sigma\to \real{S}$ is an orientation-preserving covering map. Calegari showed that
\begin{equation}\label{eqn: Admissibility}
\scl(w)\leq-\frac{\chi(\Sigma)}{2\deg(\partial\Sigma)}\,,
\end{equation}
where $\deg(\partial\Sigma)$ denotes the degree of the covering map $\partial\Sigma\to \real{S}$, for any admissible surface \cite[Proposition 2.10]{calegari_scl_2009}. If equality is realised, so
\begin{equation}\label{eqn: Extremality}
\scl(w)=-\frac{\chi(\Sigma)}{2\deg(\partial\Sigma)}\,,
\end{equation}
then $\Sigma$ is said to be \emph{extremal}. Calegari's rationality theorem \cite{calegari_stable_2009} (see also the account in \cite[Theorem 4.24]{calegari_scl_2009}) asserts that an extremal surface $\Sigma$ exists whenever $w\in [F,F]$.

A lemma of Culler \cite{culler-surfaces} (which is also implicit in the proof of Calegari's theorem) implies that $\Sigma$ has a spine $\Sigma_0$ such that the map $\Sigma\to \real{G}$ factors as the composition of a deformation retraction $\Sigma\to \real{\Sigma_0}$ with a morphism of graphs $\Sigma_0\to G$. Let $\mathbf{\widehat{\Sigma}}$ be the closed surface obtained by gluing discs to the boundary components of $\Sigma$. Then $\mathbf{\widehat{\Sigma}}$ naturally has the structure of a 2-complex $\widehat{\Sigma}$ with 1-skeleton $\Sigma_0$, and the map $\Sigma\to \real{G}$ extends to a morphism of 2-complexes $f:\widehat{\Sigma}\to X$. 

The extremality equation (\ref{eqn: Extremality}) can be rearranged to give the required relationship between total curvature and stable commutator length. Indeed, because the 1-skeleton of $\widehat{\Sigma}$ is a deformation retract of $\Sigma$,
\[
\tau(f)=\deg(f)+\chi(\Sigma)=\deg(f)(1-2\scl(w))
\]
which gives the required equation.

It now remains to show that $f$ is essential. This follows from a lemma of Calegari \cite[Lemma 2.7]{calegari_surface_2008}, which in turn relies on the subgroup separability of free groups, proved by Marshall Hall Jr \cite{hall_jr_subgroups_1949}. For completeness, we outline the proof.

Suppose that $f$ is not $\pi_1$-injective on some connected component of the 1-skeleton. Then there is a homotopically essential closed curve $\gamma$ on $\Sigma$ such that $f\circ\gamma$ is homotopically trivial in $\real{G}$. Since $\pi_1(\Sigma)$ is subgroup separable there is, for some finite $d$, a $d$-sheeted covering space $\Sigma_1\to\Sigma$ such that, after a homotopy, $\gamma$ lifts to a simple closed curve $\gamma_1$ on $\Sigma_1$. Let $\Sigma_2$ be the result of surgery on $\Sigma_1$ along $\gamma_1$; that is, $\Sigma_2$ is constructed by cutting $\Sigma_1$ along $\gamma_1$ and gluing in two discs along the resulting boundary components. Since $\gamma_1$ maps to a null-homotopic curve in $G$, $f$ extends to a map $\Sigma_2\to \real{G}$ that makes $\Sigma_2$ into an admissible surface. We now calculate:
\begin{eqnarray*}
-\frac{\chi(\Sigma_2)}{2\deg(\partial\Sigma_2)}&=&-\frac{\chi(\Sigma_1)+2}{2\deg(\partial\Sigma_1)}\\
&<&-\frac{\chi(\Sigma_1)}{2\deg(\partial\Sigma_1)}\\
&=&-\frac{d\chi(\Sigma)}{2d\deg(\partial\Sigma)}\\
&=&\scl(w)
\end{eqnarray*}
contradicting (\ref{eqn: Admissibility}). Hence, $f$ is $\pi_1$-injective on each connected component.

Now suppose that $f$ is not face-essential. Then the map from $\widehat{\Sigma}$ to its folded representative is not injective on faces, so there are distinct points $x,y\in \partial\Sigma$ with the same image in $\overline{\Sigma}_0$, the folded representative of the spine. Fix any path $\alpha$ in $\Sigma$ from $x$ to $y$. The image of $\alpha$ in $\overline{\Sigma}_0$ is a loop and so, since the map $\Sigma\to\mathbf{\overline{\Sigma}_0}$ is surjective on fundamental groups, we may concatenate $\alpha$ with a loop in $\Sigma$ so that its image in $\mathbf{\overline{\Sigma}_0}$, and hence in $G$, is null-homotopic. That is to say, we may choose $\alpha$ so that $f\circ\alpha$ is a null-homotopic loop. Next, as above, modify $\alpha$ by a homotopy and pass to a $d$-sheeted cover $\Sigma_1\to \Sigma$ so that $\alpha$ lifts to an embedded arc $\alpha_1$. 

 Let $\Sigma_2$ be constructed by attaching an interval $\beta_1$ to $\partial\Sigma_1$ at the endpoints of $\alpha_1$ and thickening it to a 1-handle. The map $f$ extends over this 1-handle and, after a homotopy to ensure that $\partial\Sigma_2\to S$ is a covering map, this makes $\Sigma_2$ into an admissible surface. Note that $\chi(\Sigma_2)=\chi(\Sigma_1)-1$ and $\deg(\partial\Sigma_2)=\deg(\partial\Sigma_1)$. The concatenation of $\alpha_1$ and $\beta_1$ defines an embedded loop $\gamma_2$ in $\Sigma_2$ such that $f\circ\gamma_2$ is null-homotopic so, as in the proof of $\pi_1$-injectivity, we may perform surgery on $\gamma_2$ to obtain a new admissible surface $\Sigma_3$. As before, we calculate:
\begin{eqnarray*}
-\frac{\chi(\Sigma_3)}{2\deg(\partial\Sigma_3)}&=&-\frac{\chi(\Sigma_2)+2}{2\deg(\partial\Sigma_2)}\\
&=&-\frac{\chi(\Sigma_1)+1}{2\deg(\partial\Sigma_1)}\\
&<&-\frac{\chi(\Sigma_1)}{2\deg(\partial\Sigma_1)}\\
&=&-\frac{d\chi(\Sigma)}{2d\deg(\partial\Sigma)}\\
&=&\scl(w)
\end{eqnarray*}
which, again, contradicts (\ref{eqn: Admissibility}). Hence, $f$ is face-essential as claimed.
\end{proof}

Corollary \ref{cor: Weak Heuer} now follows quickly from Corollary \ref{cor: Precise negative immersions statement for face-essential maps} and Proposition \ref{prop: Essential map realising scl}.

\begin{proof}[Proof of Corollary \ref{cor: Weak Heuer}]
By Theorem \ref{thm: Negative immersions}, the hypothesis that $\pi(w)>2$ is equivalent to the statement that $X$ has negative immersions. Consider the face-essential map $\widehat{\Sigma}\to X$ provided by Proposition \ref{prop: Essential map realising scl}. Since $\widehat{\Sigma}$ is a closed surface  each condition of Definition \ref{def: Visibly reducible and visibly irreducible} can be easily checked, and we see that $\widehat{\Sigma}$ is (visibly) irreducible. Combining Corollary \ref{cor: Precise negative immersions statement for face-essential maps} and Proposition \ref{prop: Essential map realising scl} then gives
\[
1-2\scl(w)=\frac{\tau(\widehat{\Sigma})}{\deg(f)}<0
\]
as required.
\end{proof}

The proposition also immediately implies that Heuer's conjecture follows from our Conjecture \ref{conj: Stability conjecture}.

\begin{corollary}\label{cor: Stability implies Heuer}
Conjecture \ref{conj: Stability conjecture} implies Conjecture \ref{conj: Heuer's conjecture}.
\end{corollary}
\begin{proof}
Let $w\in [F,F]$ be non-trivial and let $X$ be the presentation complex of the associated one-relator group $F/\ncl{w}$. Applying Conjecture \ref{conj: Stability conjecture} to the essential map $f:\widehat{\Sigma}\to X$ provided by Proposition \ref{prop: Essential map realising scl} gives
\[
1-2\scl(w)=\frac{\tau(\widehat{\Sigma})}{\deg(f)}\leq 2-\pi(w)
\] 
which rearranges to give $2\scl(w)\geq \pi(w)-1$, as required.
\end{proof}

\section{A linear system}\label{sec: A linear system}

The principal extra ingredient in the proof of Theorem \ref{thm: Uniform negative immersions} is Theorem \ref{thm: Rationality theorem}, a rationality theorem in the spirit of the main results of \cite{calegari_stable_2009} and \cite{wilton_essential_2018}. Fix a finite 2-complex $X$.  The idea is to encode face immersions from visibly irreducible complexes to $X$ as the integral points of a system of linear equations and inequations. 

\subsection{Vertex and edge pieces}

Consider a face-essential morphism $f:Y\to X$ with $Y$ visibly irreducible, and a face immersion
\[
Y\stackrel{f_0}{\to} Z\stackrel{f_1}{\to} X
\]
for $f$.  The variables of the linear system count the combinatorial types of the vertices of $Z$ and their preimages in $Y$. These are encoded as maps between (disjoint unions of) asterisk pre-complexes.

\begin{figure}[ht]
  \centerline{
    \includegraphics[width=\textwidth]{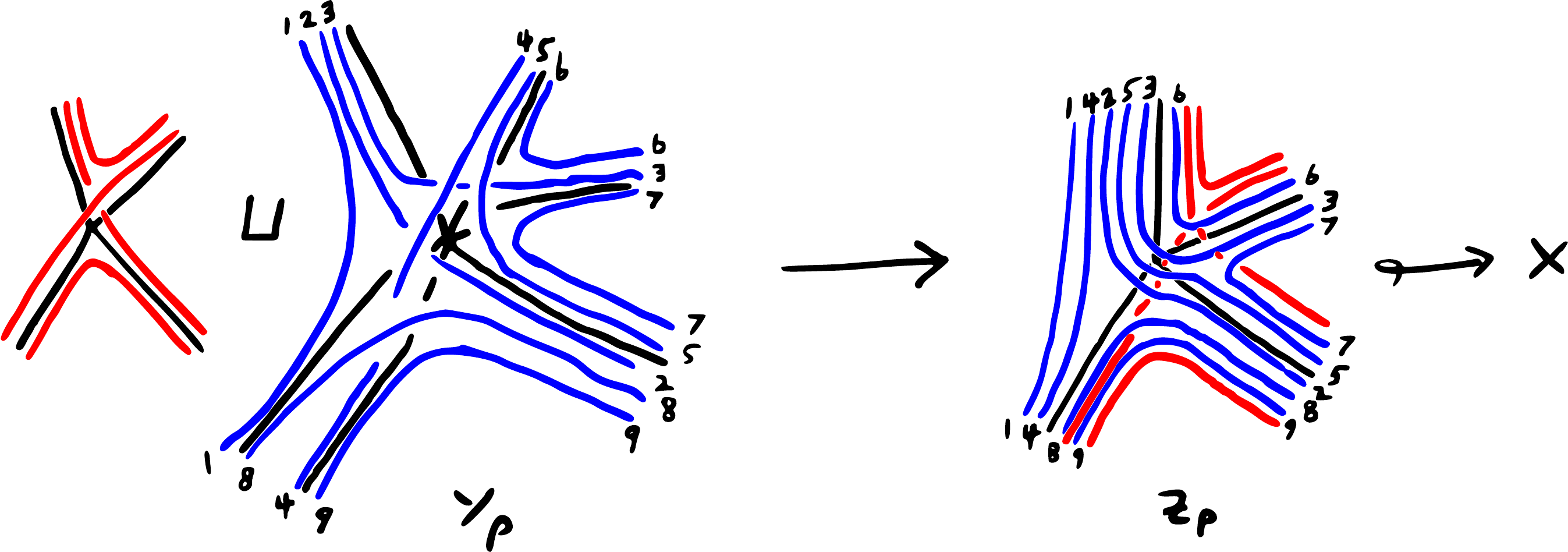}
  }
  \caption{An example of a vertex piece.}
  \label{fig: pieces}
\end{figure}

\begin{definition}\label{def: Vertex piece}
A \emph{vertex piece} $P$ over a 2-complex $X$ is a face immersion
\[
Y(P)\to Z(P)\to X
\]
where $Y(P)$ is visibly irreducible and and $Z(P)$ is an asterisk pre-complex. (In particular, $Y(P)$ is a disjoint union of asterisk pre-complexes.) Vertex pieces are considered up to the natural notion of isomorphism: if there are isomorphisms of pre-complexes $Y(P)\to Y(Q)$ and $Z(P)\to Z(Q)$ commuting with the maps $Y(\bullet)\to Z(\bullet)\to X$ then $P$ and $Q$ are isomorphic.
\end{definition}

For a 2-complex $X$, the set $\calP=\calP(X)$ of isomorphism types of vertex pieces over vertices of $X$ indexes the variables of our linear system. It is of crucial importance that there are only finitely many variables.

\begin{lemma}\label{lem: Finiteness of the set of pieces}
If $X$ is a finite complex then $\calP$ is finite.
\end{lemma}
\begin{proof}
Every regular neighbourhood $N(v)$ is finite.  Since $Z(P)\to N_X(v)$ is a branched immersion from an asterisk pre-complex it is injective, so there are only finitely many possible combinatorial types for $Z(P)$ and the map to $N_X(v)$.  Since $Y(P)\to Z(P)$ is a face-equivalence, the cardinality of $S_{Y(P)}$ is also bounded. Finally, since $Y(P)$ is visibly irreducible, each edge and vertex of $Y(P)$ has at least two pre-images in $S_{Y(P)}$, so the cardinality of $G_{Z(P)}$ is also bounded. So there are at most boundedly many different possible isomorphism types for $Y(P)$, and hence also only boundedly many possible maps $Y(P)\to Z(P)$.
\end{proof}

The vertex pieces will define the variables of our linear system.  Edge pieces will give the defining equations. An \emph{edge pre-complex} is a pre-complex whose underlying pre-graph consists of a single edge and no vertices.

\begin{definition}\label{def: Edge piece}
An \emph{edge piece} $R$ over an edge $e$ of a 2-complex $X$ consists of a face immersion
\[
Y(R)\to Z(R)\to X
\]
where $Z(R)$ is an edge pre-complex. Again, edge pieces are considered up to the natural notion of isomorphism.
\end{definition}

Let $\mathcal{E}=\mathcal{E}(X)$ denote the set of isomorphism types of edge pieces over edges of $X$.

\subsection{Weight vectors}

Fix a finite 2-complex $X$. Again, consider a face immersion
\[
Y\stackrel{f_0}{\to} Z\stackrel{f_1}{\to} X
\]
for a face-essential morphism $f:Y\to X$ with $Y$ visibly irreducible.  Although this face immersion is our actual object of study, we will often abuse notation and just denote it by $f$, or even $Y$.  

\begin{definition}\label{def: Induced pieces}
Each vertex $v$ of $Z$ defines an \emph{induced vertex-piece} $P\equiv P_Y(v)$ over $f_1(v)$ as follows:
\begin{enumerate}[(i)]
\item $Z(P):= N_{Z}(v)$;
\item $Y(P):= \coprod_{f_0(u)=v} N_Y(u)$.
\end{enumerate}
The morphisms $Y(P)\to Z(P)\to X$ are the natural induced morphisms.  Each edge $e$ of $Z$ defines an \emph{induced edge-piece} $R_Y(e)$ in a similar manner.
 \end{definition}

 We translate face immersions into linear algebra by counting the number of times that each vertex-piece appears.  This count naturally lives in the following vector space.
 
\begin{definition}\label{def: Space of weight vectors}
 The \emph{space of weight vectors} $\RR[\calP]$ is the real vector space of formal linear sums 
\[
v=\sum_{P\in\calP} v_P P
\]
of vertex-pieces $P\in\calP$; the elements of this vector space are called \emph{weight vectors}. We will also talk about the additive subgroups $\QQ[\calP]$ and $\ZZ[\calP]$ consisting of rational and integral points respectively.
\end{definition}

There is a map that associates to the face immersion
\[
Y\stackrel{f_0}{\to} Z\stackrel{f_1}{\to} X
\]
from a visibly irreducible complex $Y$ a weight vector:
\[
v(Y):=\sum_{y\in V_Z} P_Y(y)\,,
\]
recalling that $V_Z$ is the set of vertices of the 1-skeleton of $Z$. Clearly, $v(Y)$ has {non-negative} integral coordinates, {and, in fact}, the possible weight vectors are exactly the points with { non-negative integral coordinates} that satisfy a certain system of linear equations, called the \emph{gluing equations}. These gluing equations make use of \emph{boundary maps}, which associate to each vertex piece the induced edge-pieces of the incident vertices.  Let $\RR[\calE]$ be the real vector space of formal linear sums of edge pieces.

\begin{definition}\label{def: Boundary maps and gluing equations}
Recall from Definition \ref{def: Induced pieces} that, for any vertex piece
\[
Y(P)\to Z(P)\to X
\]
and any edge $e$ of $Z(P)$, there is an induced edge piece $R_{Y(P)}(e)$.

The boundary map $\partial^\iota:\RR[\calP] \to \RR[\calE]$ is defined by extending
\[
\partial^\iota\negthinspace(P):=\sum_{e\in E^\iota_{Z(P)}} R_{Y(P)}(e)
\]
linearly.  Similarly, the assignment
\[
\partial^\tau\negthinspace(P):=\sum_{e\in E^\tau_{Z(P)}} R_{Y(P)}(e)
\]
extends linearly to define the boundary map $\partial^\tau:\RR[\calP] \to \RR[\calE]$. We then set
\[
\partial:=\partial^\iota-\partial^\tau\,.
\]
A vector $v\in\RR[\calP]$ is said to \emph{satisfy the gluing equations} if $\partial(v)=0$.
\end{definition}

The domain of the optimisation problem of interest to us is the cone of { non-negative} vectors that satisfy the gluing equations. To this end, set
\[
\RR_{\geq 0}[\calP]:=\left\{\sum_{P\in\calP} v_P P\,\middle|\, v_P\geq 0\textrm{~for\,all\,}P\right\}\,.
\] 
We are interested in the cone
\[
C(\RR):=\ker{\partial}\cap\RR_{\geq 0}[\calP]\,,
\]
and also in the sets of integer points $C(\ZZ):=C(\RR)\cap \ZZ[\calP]$ and of rational points $C(\QQ):=C(\RR)\cap \QQ[\calP]$. The key fact here is that the map to weight vectors surjects the integer points of $C(\RR)$. It is also finite-to-one, although we will not make use of that here.

\begin{proposition}\label{prop: Surjectivity of map to weight vectors}
Let $X$ be a finite 2-complex. If $Y$ is a finite, visibly irreducible complex, $Z$ is also a complex and
\[
Y\to Z\to  X
\]
is a face immersion then $v(Y)\in C(\ZZ)$. Conversely, if $v\in C(\ZZ)$ then there is a face immersion from a finite visibly irreducible complex
\[
Y\to Z\to X
\]
such that $v=v(Y)$.
\end{proposition}
\begin{proof}
Consider an edge piece $R\in\calE$ over $X$.  The $R$-coefficient of $\partial^{\iota}\negthinspace(v(Y))$ counts the number of vertices $z$ of $Z$ for which there is an edge $e$ with $\iota(e)=z$ such the induced edge piece $R_Y(e)$ is isomorphic to $R$.  Since $Z$ is a complex, such vertices $z$ are in bijection with such edges $e$, and so this coordinate is equal to the number of edges $e$ of $Z$ such that $R_Y(e)$ is isomorphic to $R$: notice that the latter quantity does not mention $\iota$. Similarly, the $R$-coefficient of $\partial^{\tau}\negthinspace(v(Y))$ also counts the number of such edges $e$, and so $\partial^\iota\negthinspace(v(Y))=\partial^{\tau}\negthinspace(v(Y))$. Therefore, $v(Y)\in C(\ZZ)$ as claimed.

To prove surjectivity, consider an integral weight vector $v\in C(\ZZ)$, and define the pre-complexes
\[
Z':=\coprod_{P\in\calP}\coprod_{i_P=1}^{v_P} Z(P)
\] 
and
\[
Y':=\coprod_{P\in\calP}\coprod_{i_P=1}^{v_P} Y(P)\,,
\]
noting that these come naturally equipped with maps to define a face immersion
\[
Y'\to Z'\to X
\]
Furthermore, $Y'$ is visibly irreducible by construction.

The gluing equations imply that for each edge piece $R$ there is a bijection between the set of edges $e\in E^{\iota}_{Z'}$ such that $R_{Y'}(e)=R$ and the set of $e\in E^{\tau}_{Z'}$ such that $R_{Y'}(e)=R$.  Therefore, we may choose a bijection $F:E^{\iota}_{Z'}\to E^{\tau}_{Z'}$ such that $R_{Y'}(F(e))=R_{Y'}(e)$. Inductively applying the gluing construction from Definition \ref{def: Gluing}, $Z'$ can be glued up to a complex $Z$ equipped with a branched immersion $Z\to X$ and a branched map $Y'\to Z$.

It remains to explain how to glue up $Y'$ to create complex $Y$.  For an edge $e$ of $Z$, let $P(\iota(e))$ denote the piece used to construct the vertex $\iota(e)$ of $Z$, and let $P(\tau(e))$ the piece used to construct the vertex $\tau(e)$ of $Z$.  The pre-complexes $Y(P(\iota(\epsilon)))$ and $Y(P(\tau(\epsilon)))$ each naturally induce partitions of $w_Z^{-1}(e)$, and by construction, these two partitions are equal. Therefore, there is a bijection $G:E^{\iota}_{Y'}\to E^{\tau}_{Y'}$ that respects the map $Y'\to Z$. Again, inductively applying the gluing construction from Definition \ref{def: Gluing}, $Y'$ can be glued up to a complex $Y$ equipped with face-equivalence $Y\to Z$.  By construction, $Y$ is visibly irreducible. 
\end{proof}

\subsection{Degree and curvature}

The final ingredient is to notice that, given a face-essential map $f:Y\to X$, the total curvature and the degree are both realised by linear functions on $C(\RR)$.

\begin{definition}\label{def: Degree as a linear function}
Let $l_X(C)$ denote the number of vertices in each component $C$ of $S_X$. Recall that a vertex-piece $P\in\calP$ consists of a face immersion
\[
Y(P)\to Z(P)\to X\,.
\]
In particular, each vertex $x$ of $S_{Y(P)}$ is sent to some component $C_x$ of $S_X$. The degree map $\deg:\RR[\calP]\to\RR$ is now given by linearly extending the assignment
\[
\deg(P):=\sum_{x\in V_{S_{Y(P)}}}\frac{1}{l_X(C_x)}\,.
\]
\end{definition}

The point of this definition is to extend the definition of the degree of a face-essential map.

\begin{lemma}\label{lem: Degree is linear}
If 
\[
Y\to Z\to X
\]
is a face immersion for a face-essential map $f:Y\to X$, then
\[
\deg(v(Y))=\deg(f)\,.
\]
\end{lemma}
\begin{proof}
For any component $C$ of $S_Y$, $l_Y(C)=\deg_C(f)l_X(f(C))$. Therefore,
\begin{eqnarray*}
\deg(f)&=&\sum_{[C]\in\pi_0(S_Y)} \deg_C(f)\\
&=& \sum_{[C]\in\pi_0(S_Y)} \frac{l_Y(C)}{l_X(f(C))}\\
&=&\sum_{[C]\in\pi_0(S_Y)}\sum_{x\in V_C} \frac{1}{l_X(f(C))}\\
&=&\sum_{x\in V_{S_Y}} \frac{1}{l_X(C_x)}\\
&=&\sum_{P\in\calP}\sum_{P_Y(y)=P}\deg(P)\\
&=&\deg(v(Y))
\end{eqnarray*}
as required.
\end{proof}

A similar observation applies to the total curvature.

\begin{definition}\label{def: Total curvature as a linear function}
For a vertex-piece $P\in\cal P$, set
\[
\tau(P):=\deg(P)+\#V_{Y(P)}-\#E_{Y(P)}/2
\]
and extend this to a linear function $\tau:\RR[\calP]\to \RR$. 
\end{definition}

Again, this extends the total curvature to a linear function on $C(\RR)$.

\begin{lemma}\label{lem: Total curvature is linear}
If 
\[
Y\to Z\to X
\]
is a face immersion for a face-essential map $f:Y\to X$, then $\tau(v(Y))=\tau(Y)$.
\end{lemma}
\begin{proof}
By definition, $\tau(Y)=\deg(f)+\chi(G_Y)$, so it suffices to prove that 
\[
\chi(G_Y)=\sum_{v\in V_Z}(\#V_{Y(P(v))}-\#E_{Y(P(v))}/2)\,.
\]
Since $V_{Y(P(v))}$ bijects with the set of vertices of $Y$ that map to $v\in Z$,
\[
\sum_{v\in V_Z}\#V_{Y(P(v))}=\# V_Y\,.
\]
Likewise, $E_{Y(P(v))}$ bijects with the set of edges of $Y$ incident at a vertex mapping to $v$, so 
\[
\sum_{v\in V_Z}\#E_{Y(P(v))}=2\# E_Y\,.
\]
Therefore, 
\[
\sum_{v\in V_Z}(\#V_{Y(P(v))}-\#E_{Y(P(v))}/2)=\# V_Y-\#E_Y=\chi(G_Y)
\]
as required.
\end{proof}

The following rationality theorem now follows by standard optimisation techniques.

\begin{theorem}\label{thm: Rationality theorem}
Let $X$ be a finite 2-complex. There is a finite, visibly irreducible complex $Y_{\max}$ and a face-essential map $f_{\max}:Y_{\max} \to X$ such that
\[
\frac{\tau(Y_{\max})}{\deg(f_{\max})}\geq\frac{\tau(Y)}{\deg(f)}
\]
for all face-essential maps $f:Y\to X$ with $Y$ finite and visibly irreducible.
\end{theorem}
\begin{proof}
Consider the rational polyhedron
\[
\Delta:=C(\RR)\cap \{ \deg(v)=1\}\,.
\]
Every coordinate vector in $\RR[\calP]$ has positive degree, and so
\[
\RR_{\geq 0}[\calP]\cap \{ \deg(v)=1\}
\]
is a simplex. Hence, $\Delta$ is compact.

By the simplex algorithm, the linear map $\tau$ is maximised at some vertex $v_{\max}$ of $\Delta$. Since $\Delta$ is rationally defined, $v_{\max}\in C(\QQ)$, and so has some multiple $u_{\max}\in C(\ZZ)$. By Proposition \ref{prop: Surjectivity of map to weight vectors}, $u_{\max}=v(Y_{\max})$ for some face immersion
\[
Y_{\max}\to Z_{\max}\to X\,,
\]
with {$Y_{\max}$} finite and visibly irreducible.  In particular, by Lemma \ref{lem: Face-essential maps and face immersions}, the morphism $f_{\max}:Y_{\max}\to X$ is face-essential.

For any face-essential map $f:Y\to X$, the factorisation through the folded representative defines a face immersion
\[
Y\to\bar{Y}\to X
\]
by Lemma \ref{lem: Face-essential maps and face immersions}, so $v(Y)/\deg(f)\in \Delta$ and
\[
\frac{\tau(Y_{\max})}{\deg(f_{\max})}=\tau(v_{\max})\geq \frac{\tau(Y)}{\deg(f)}
\]
as required.
\end{proof}

Combined with Corollary \ref{cor: Precise negative immersions statement for face-essential maps}, we can complete the proof of Theorem \ref{thm: Uniform negative immersions}.

\begin{proof}[Proof of Theorem \ref{thm: Uniform negative immersions}]
Let $f_{\max}:Y_{\max}\to X$ be the face-essential map provided by Theorem \ref{thm: Rationality theorem}, and set
\[
\epsilon:=-\frac{\tau(Y_{\max})}{\deg(f_{\max})}\,.
\]
Since $X$ has negative immersions and $Y_{\max}$ is visibly irreducible, Corollary \ref{cor: Precise negative immersions statement for face-essential maps} applies to show that $\tau(Y_{\max})<0$, whence $\epsilon>0$ as required.  

Suppose now that $Y_0$ is a finite, irreducible complex and that $f_0:Y_0\to X$ is an immersion; in particular, $\tau(Y_0)=\chi(Y_0)$ and $\deg(f)$ is the number of faces of  $Y_0$. By Lemma \ref{lem: Unfolding all the way}, $Y_0$ unfolds to a finite, visibly irreducible 2-complex $Y\to Y_0$,  and the composition
\[
f:Y\to Y_0\stackrel{f_0}{\to} X
\]
is face-essential. Therefore, 
\begin{eqnarray*}
-\epsilon&\geq& \frac{\tau(Y)}{\deg(f)}\\
&=&\frac{\tau(Y_0)}{\deg(f_0)}\\
&=& \frac{\chi(Y_0)}{{\#\{2\mathrm{-cells\,of}\,Y\}}}
\end{eqnarray*}
as required.
\end{proof}

\section{The case with torsion}\label{sec: The case with torsion}

We shall see that Theorem \ref{thm: Uniform negative immersions} leads to strong constraints on the subgroup structure of one-relator groups with negative immersions. As stated, it only applies to one-relator groups without torsion. Baumslag's coherence conjecture  was proved in the case with torsion independently by the authors \cite{louder_one-relator_2020} and Wise \cite{wise_coherence_2020}.  The purpose of this section is to note that one-relator groups with torsion also fit into the framework of this paper.

The presentation complex of a one-relator group with torsion also has uniform negative immersions, after passing to a finite-sheeted covering space\footnote{Alternatively, one might work with a presentation orbi-complex, which itself has uniform negative immersions for maps from complexes.}. 

\begin{theorem}\label{thm: Uniform negative immersions in the torsion case}
 Consider a one-relator group with torsion $G=F/\ncl{u^n}$, where $n>1$. There is a torsion-free subgroup $G_0$ of finite index in $G$ that is the fundamental group of a finite 2-complex $X_0$ with uniform negative immersions.
 \end{theorem}
 \begin{proof}
By \cite[Theorem 2.2]{louder_one-relator_2020}, there is a torsion-free subgroup $G_0$ of $G$ that is the fundamental group of a complex $X_0$.  As in the proof of \cite[Theorem 2.2]{louder_one-relator_2020}, $X_0$ can be obtained from the presentation complex $X$ of $G$ by passing to a finite-sheeted cover and collapsing duplicate faces.  By \cite[Corollary 3.1]{louder_one-relator_2020}, any immersion $Y\immerses X_0$ satisfies
\[
\chi(Y)+(n-1){\#\{2\mathrm{-cells\,of}\,Y\}}\leq 0
\]
as long as $Y$ is finite and irreducible. Rearranging, we obtain that
\[
\frac{\chi(Y)}{\#\{2\mathrm{-cells\,of}\,Y\}}\leq 1-n<0\,,
\]
so $X_0$ has uniform negative immersions as claimed.
 \end{proof}

\section{Uniform negative immersions and subgroups}\label{sec: Uniform negative immersions and subgroups}

In this section we complete the proof of Theorems \ref{thm: Coherence} and \ref{thm: Other subgroup constraints}, by proving these properties for the fundamental group of any complex $X$ with uniform negative immersions. Our strategy is similar to the proof of coherence for one-relator groups with torsion  \cite{louder_one-relator_2020,wise_coherence_2020}. There is, however, a small extra technicality, because we need to deal with irreducible 2-complexes, rather than {non-visibly-reducible} 2-complexes; this extra technicality is resolved by Lemmas \ref{lem: Irreducible core} and \ref{lem: Universal property of irreducible cores}.

Let $G$ be a one-relator group with negative immersions, let $X$ be its standard presentation complex, and let $H$ be a finitely generated subgroup of $G$.  The next lemma is the key consequence of uniform negative immersions that we will make use of.  Two immersions of 2-complexes $Y_1\to X$ and $Y_2\to X$ are \emph{isomorphic} if there is an isomorphism of 2-complexes $Y_1\to Y_2$ such that the obvious diagram commutes.

\begin{lemma}\label{lem: Bound on immersions}
Let $X$ be a finite 2-complex with uniform negative immersions. For any integer $r$, there are only finitely many isomorphism classes of immersions $Y\to X$ such that $Y$ is finite, connected, irreducible, and $b_1(Y)\leq r$ (where $b_1(Y)$ denotes the first Betti number of $Y$).
\end{lemma}
\begin{proof}
Consider such an immersion $Y\immerses X$. Rearranging the defining inequality of uniform negative immersions, we have that
\[
\#\{2\mathrm{-cells\,of}\,Y\}\leq \frac{-\chi(Y)}{\epsilon}\,.
\]
But 
\[
\chi(Y)\geq 1-b_1(Y)\geq 1-r
\]
since $Y$ is a 2-complex, so $\#\{2\mathrm{-cells\,of}\,Y\}\leq (r-1)/\epsilon$.  Thus, there is a uniform bound on the number of faces of $Y$. Because $Y$ is irreducible, every vertex or edge of $Y$ is incident at a face. Since $Y\immerses X$ is an immersion, there is a bound on the number of faces at which each vertex and edge of $Y$ are incident, so the bound on the number of faces implies bounds on the numbers of vertices and edges. Therefore, there are only finitely many possible combinatorial types for $Y$, and also for immersions $Y\immerses X$.
\end{proof}

Note that Lemma \ref{lem: Bound on immersions} fails for 2-complexes with non-positive immersions. For instance, the torus has infinitely many non-isomorphic finite-sheeted covering spaces, all with fundamental groups generated by $2$ elements.

Now consider a finitely generated subgroup $H$ of $G=\pi_1(X)$. By Grushko's theorem, we may assume that $H$ is freely indecomposable. The next lemma is very similar to \cite[Lemma~4.2]{louder_one-relator_2020}; however, the reader should beware that the definition of irreducible complex used here is more restrictive than the definition in \cite{louder_one-relator_2020}.

\begin{lemma}\label{lem: Irreducible folding sequence}
Let $X$ be a 2-complex, $G=\pi_1(X)$, and $H\leq G$ a finitely generated, non-cyclic, freely indecomposable subgroup. After replacing $H$ by a conjugate, there is a sequence of $\pi_1$-surjective immersions of compact, connected 2-complexes
\[
Y_0\immerses Y_1\immerses Y_2\immerses\cdots\immerses Y_i\immerses\cdots X
\]
with the following properties:
\begin{enumerate}[(i)]
\item each $Y_i$ is irreducible;
\item $H=\underrightarrow{\lim}\,\pi_1Y_i$.
\end{enumerate}
\end{lemma}

Because of the stronger notion of irreducibility used in this paper, the proof of Lemma \ref{lem: Irreducible folding sequence} requires one extra ingredient: the irreducible core of Lemma \ref{lem: Irreducible core}. With this in hand, the proof of Lemma \ref{lem: Irreducible folding sequence} is very similar to the proof of \cite[Lemma~{4.3}]{louder_one-relator_2020}.

\begin{proof}[Proof of Lemma \ref{lem: Irreducible folding sequence}]
\cite[Lemma~{4.3}]{louder_one-relator_2020} (see also Remark \ref{rem: Immersions in orbicomplexes vs complexes}) provides a sequence of $\pi_1$-surjective immersions of $2$--complexes
\[
Z_i\to Z_{i+1}\stackrel{f_{i+1}}{\to} X
\]
such that $H=\underrightarrow{\lim}\,\pi_1Z_i$. It remains to improve this to make the terms irreducible.

Since $H$ is non-cyclic, no $Z_i$ has cyclic fundamental group.   By Scott's lemma (proved with extra hypotheses as \cite[Lemma 2.2]{scott_finitely_1973}, proved as \cite[Theorem 2.1]{swarup_delzant_2004}), since $H$ is freely indecomposable, after throwing away finitely many terms from the sequence we may assume that each $\pi_1(Z_i)$ is freely indecomposable.

Applying Lemmas \ref{lem: Irreducible core} and \ref{lem: Universal property of irreducible cores} inductively constructs a second sequence
\begin{center}
  \begin{tikzcd}
Y_0\arrow{r}\arrow{d} & Y_1\arrow{r}\arrow{d}& Y_2\arrow{r}\arrow{d}& \cdots\arrow{r} & Y_i \arrow{r}\arrow{d} &\cdots &\\
Z_0\arrow{r} &Z_1\arrow{r}& Z_2\arrow{r}& \cdots\arrow{r}& Z_i \arrow{r} &\cdots \arrow{r}& X
  \end{tikzcd}
\end{center}
where each $Y_i$ is the irreducible core $\widehat{Z}_i$ of $Z_i$. Now choose  a basepoint in $Y_0$, if necessary moving the basepoints of the $Z_i$ --- this may involve conjugating $H$. The result now follows.
\end{proof}

The next theorem, which is an easy consequence of Lemmas \ref{lem: Bound on immersions} and \ref{lem: Irreducible folding sequence}, will imply Theorem \ref{thm: Coherence}. Its proof is an elaboration of the arguments of \cite{louder_one-relator_2020} or \cite{wise_coherence_2020}.

\begin{theorem}\label{thm: Realising subgroups by immersions}
Let $X$ be a finite 2-complex with uniform negative immersions, and let $G=\pi_1(X)$. For any finitely generated, non-cyclic, freely indecomposable subgroup $H$ of $G$, there is an immersion from a finite, irreducible 2-complex $Y$ to $X$ such that $\pi_1(Y)\cong H$ and $Y\immerses X$ induces the inclusion of $H$ into $G$ up to conjugacy.
\end{theorem}
\begin{proof}
After possibly conjugating $H$, Lemma \ref{lem: Irreducible folding sequence} provides a sequence of $\pi_1$-surjective immersions of compact, connected, irreducible complexes
\[
Y_0\immerses Y_1\immerses Y_2\immerses\cdots\immerses Y_i\immerses\cdots X
\]
such that $H=\underrightarrow{\lim}\,\pi_1Y_i$. Since the maps $Y_i\to Y_{i+1}$ are surjective on fundamental groups, there is a uniform bound
\[
b_1(Y_0)\geq b_1(Y_i)
\]
on the first Betti numbers. Therefore, by Lemma \ref{lem: Bound on immersions}, infinitely many of the immersions $Y_i\immerses X$ are isomorphic, so passing to a subsequence we may assume that they are all equal to a fixed immersion $Y\immerses X$.  By \cite[Lemma 6.3]{wise_sectional_2004}, each immersion $Y_i\immerses Y_{i+1}$ is an isomorphism. In particular, since $H=\underrightarrow{\lim}\,\pi_1(Y_i)$, every $Y\to X$ represents $H$, as required.
 \end{proof}

Coherence is an immediate consequence.

\begin{corollary}\label{cor: Uniform negative immersions implies coherence}
If $X$ is a finite 2-complex with uniformly negative immersions then $\pi_1(X)$ is coherent.
\end{corollary}
\begin{proof}
Let $H$ be a finitely generated subgroup of $G$.  By Grushko's theorem, $H$ can be written as a free product
\[
H=F*H_1\ldots * H_n
\]
where $F$ is free and each $H_i$ is freely indecomposable. Each factor $H_i$ is finitely presented by Theorem \ref{thm: Realising subgroups by immersions}, so the result follows.
\end{proof}

Theorem \ref{thm: Coherence} now follows quickly.

\begin{proof}[Proof of Theorem \ref{thm: Coherence}]
{The theorem} follows by combining Theorems \ref{thm: Negative immersions} and \ref{thm: Uniform negative immersions} with Corollary \ref{cor: Uniform negative immersions implies coherence}.
\end{proof}

Likewise, the conclusions of Theorem \ref{thm: Other subgroup constraints} also hold for all complexes with uniform negative immersions.

\begin{theorem}
\label{thm: Other subgroup constraints for UNI complexes}
 Let $X$ be a finite 2-complex with uniform negative immersions, and let $G=\pi_1(X)$.
  \begin{enumerate}[(i)]
   \item Every finitely generated {one-ended} subgroup $H$ of  $G$ is co-Hopfian, i.e.\ $H$ is not isomorphic to a proper subgroup of itself. In particular, $G$ itself is co-Hopfian.
   \item For any integer $r$, there are only finitely many conjugacy classes of {finitely generated one-ended} subgroups $H$ of $G$  such that the abelianisation of $H$ has rational rank at most  $r$.
    \item Every finitely generated non-cyclic subgroup $H$ of $G$ is large in the sense of Pride, i.e.\ there is a subgroup $H_0$ of finite index in $H$ that surjects a non-abelian free group.
\end{enumerate}
\end{theorem}
\begin{proof}
Let $H$ be a non-cyclic, finitely generated subgroup of $G$.  

To prove item (i), assume that $H$ is freely indecomposable, so is realised up to conjugacy by an immersion $Y=Y_0\immerses X$, by Theorem \ref{thm: Realising subgroups by immersions}.  Given any injective homomorphism $H\to H$, applying Theorem \ref{thm: Realising subgroups by immersions} inductively gives rise to a sequence of immersions of finite, connected, irreducible complexes
\[
\ldots\immerses Y_i\immerses\ldots \immerses Y_1\immerses Y_0\immerses X\,,
\]
where $\pi_1(Y_i)\cong H$ for all $i$.  By Lemma \ref{lem: Bound on immersions}, there are finitely many isomorphism classes of complexes $Y_i$ so, passing to a subsequence, we may assume that $Y_i=Y$ for some fixed $Y$ and all $n$.  Since $Y$ is compact, every immersion $Y\immerses Y$ is an isomorphism \cite[Lemma 6.3]{wise_sectional_2004}, so the homomorphism $H\to H$ is also surjective, as required. 

Item (ii) is an immediate consequence of Theorem \ref{thm: Realising subgroups by immersions} and Lemma \ref{lem: Bound on immersions}.

Finally, to prove item (iii), start by assuming that $H$ is freely indecomposable.  By Theorem \ref{thm: Realising subgroups by immersions}, up to conjugacy, the inclusion of $H$ into $G$ is induced by an immersion $Y\immerses X$, where $Y$ is a finite, connected, irreducible 2-complex.  By negative immersions, $\chi(Y)\leq -1$. A choice of maximal tree in the 1-skeleton of $Y$ leads to a presentation for $H$ with $m$ generators and $n$ relators, where $\chi(Y)=1-m+n$, so the deficiency of this presentation is
\[
m-n=1-\chi(Y)\geq 2\,.
\]
Therefore, by the Baumslag--Pride theorem \cite{baumslag_groups_1978}, $H$ is large, as required. For general $H$, consider the Grushko decomposition
\[
H=F*H_{1}*\ldots*H_k
\]
where $F$ is free and each $H_i$ is non-cyclic and freely indecomposable. Since the factors $H_i$ are all large, it follows that $H$ is large unless $H$ is cyclic.
\end{proof}

Theorem \ref{thm: Other subgroup constraints} follows immediately from Theorems \ref{thm: Uniform negative immersions} and \ref{thm: Other subgroup constraints for UNI complexes}.

We close by noting that Lemma \ref{lem: Bound on immersions} above can be made somewhat effective.

\begin{proposition}\label{prop: Effective bound on number of 2-cells}
Let $X$ be the presentation complex of  a one-relator group $G=F/\ncl{w}$ with negative immersions. There is an algorithm that takes $w$ as input and outputs a constant $C=C(w)$ such that every one-ended subgroup generated $H$ of $G$ with abelianisation of rational rank at most $r$ is represented by an immersion $Y\immerses X$ with at most $C(r-1)$ 2-cells.
\end{proposition}
\begin{proof}
Theorem \ref{thm: Uniform negative immersions} asserts that $X$ has uniform negative immersions, meaning that there is an $\epsilon>0$ such that
\[
\frac{\chi(Y)}{\#\{2\mathrm{-cells\,of}\,Y\}}\leq -\epsilon\,,
\]
for every irreducible $Y$ immersing into $X$. The constant $\epsilon$ is an extremal value of the linear-programming problem used in the proof of Theorem \ref{thm: Rationality theorem}. This linear-programming problem is explicitly defined in terms of $X$, so $\epsilon$ can be computed from the data of $X$. The proof of Lemma \ref{lem: Bound on immersions} now shows that the number of 2-cells of $Y$ is at most $(r-1)/\epsilon$, so we may take $C=1/\epsilon$.
\end{proof}

Note that Proposition \ref{prop: Effective bound on number of 2-cells} does not immediately imply that there is an algorithm to compute presentations for finitely generated subgroups of one-relator groups. In \cite{groves_enumerating_2009}, a coherent group in which there is an algorithm to compute a finite presentation for a given finitely generated subgroup is called \emph{effectively coherent}. The following question remains open.

\begin{question}\label{qu: Effective coherence}
Are one-relator groups with negative immersions effectively coherent?
\end{question}

\bibliographystyle{amsalpha}
\providecommand{\bysame}{\leavevmode\hbox to3em{\hrulefill}\thinspace}
\providecommand{\MR}{\relax\ifhmode\unskip\space\fi MR }
\providecommand{\MRhref}[2]{%
  \href{http://www.ams.org/mathscinet-getitem?mr=#1}{#2}
}
\providecommand{\href}[2]{#2}

\Addresses

\end{document}